 \documentclass[11pt]{amsart}

\usepackage[margin=3.5cm]{geometry}

\usepackage[T2A,OT1]{fontenc}
\usepackage[ot2enc]{inputenc}
\usepackage[russian,english]{babel}
\usepackage{amsmath,amsthm,amscd,amssymb,latexsym}
\usepackage{tabularx}
\usepackage{rotating}
\usepackage[all]{xy}
\usepackage{hypbmsec}

\theoremstyle{plain}
\newtheorem{theorem}{Theorem}[section]
\newtheorem{lemma}{Lemma}[section]

\newtheorem{proposition}{Proposition}[section]
\theoremstyle{definition}
\newtheorem{definition}{Definition}
\newtheorem{conjecture}{Conjecture}

\newtheorem{problem}{Problem}
\newtheorem{remark}{Remark}[section]
\numberwithin{equation}{section}

\newcommand{\bbE}{{\mathbb{E}}}
\newcommand{\bbD}{{\mathbb{D}}}

\newcommand{\bbR}{{\mathbb{R}}}

\newcommand{\bbZ}{{\mathbb{Z}}}
\newcommand{\bbC}{{\mathbb{C}}}

\newcommand{\bbT}{{\mathbb{T}}}

\newcommand{\cA}{{\mathcal{A}}}

\newcommand{\cD}{{\mathcal{D}}}
\newcommand{\cE}{{\mathcal{E}}}
\newcommand{\cF}{{\mathcal{F}}}

\newcommand{\cI}{{\mathcal{I}}}

\newcommand{\cN}{{\mathcal{N}}}
\newcommand{\cM}{{\mathcal{M}}}

\newcommand{\cP}{{\mathcal{P}}}

\newcommand{\cR}{{\mathcal{R}}}

\newcommand{\cZ}{{\mathcal{Z}}}
\newcommand{\cX}{{\mathcal{X}}}
\newcommand{\cY}{{\mathcal{Y}}}
\newcommand{\cW}{{\mathcal{W}}}
\newcommand{\cV}{{\mathcal{V}}}

\newcommand{\fX}{{\mathfrak{X}}}

\newcommand{\fj}{{\mathfrak{j}}}
\newcommand{\fm}{{\mathfrak{m}}}
\newcommand{\fp}{{\mathfrak{p}}}

\newcommand{\fw}{{\mathfrak{w}}}

\newcommand{\G}{\Gamma}

\newcommand{\e}{{\varepsilon}}

\newcommand{\vt}{\vartheta}
\newcommand{\vpi}{\varpi}

\newcommand{\z}{\zeta}
\newcommand{\g}{\gamma}
\newcommand{\Up}{\Upsilon}
\renewcommand{\l}{\lambda}

\newcommand{\pd}{{\partial}}

\newcommand{\clos}{\text{\rm clos}\,}

\newcommand{\oc}{\overset{\circ}}

\renewcommand{\Re}{\text{\rm Re}\,}
\renewcommand{\Im}{\text{\rm Im}\,}

\AtBeginDocument{%
	\def\MR#1{}
}

\title{Ahlfors problem for polynomials}

\author[Benjamin~Eichinger]{B.~Eichinger}
\address{Institute for Analysis, Johannes Kepler University of Linz,
	4040 Linz, Austria}
\email{benjamin.eichinger@jku.at}

\author[Peter~Yuditskii]{P.~ Yuditskii}
\address{
	Abteilung f\"ur Dynamische Systeme,
	und Approximationstheorie,
	Johannes Kepler Universit\"at Linz, 4040 Linz, Austria}
\email{petro.yudytskiy@jku.at}

\thanks{Both authors were supported by the Austrian Science Fund FWF, project no: P25591-N25.}

\date{\today}


\begin{document}

\begin{abstract}
We raise a conjecture that asymptotics for Chebyshev polynomials in a complex domain can be given in terms of  the reproducing kernels of a suitable Hilbert space of analytic functions in this domain. It is based on two classical results due to Garabedian and Widom. To support this conjecture we study asymptotics for Ahlfors extremal polynomials in the complement to a system of intervals on $\bbR$, arcs on $\bbT$, and its continuous counterpart. 

Bibliography: 34~titles.
\end{abstract}

\maketitle

\section{Introduction}
\label{s1}
Starting from works of Chebyshev, Markov, Bernstein, Akhiezer, Widom, ... many explicit asymptotics for the best uniform approximation were found and became classical, see e.g., the book \cite{Akh53}, especially addendum in it, see also a survey paper \cite{SY92}. As soon as a problem deals with approximation on the real axis
we have results in a really very broad range, including, say, approximations with 
varying weights \cite{LL01}, see also \cite{Su66}, or approximations of quite exotic functions, e.g. \cite{ErY11}, which by the way, has applications in computational mathematics \cite{ZY16}.
Most likely, we owe it to  the Chebyshev alternation theorem, which  reveals  completely the structure of generalized polynomials of minimal deviation from zero on the real axis.  

Despite some new results of the highest level, e.g. \cite{Tot12,KNT16},
in the \textit{complex plane} not much is known even in the most classical setting. We do not have asymptotics for \textit{Chebyshev polynomials}
in finitely connected  domains bounded by smooth  arcs, (or even in simply connected domains, for instance, in the complement of a spiral curve).

The main goal of this paper is to give some reasons to support the following hypothesis. 

\begin{conjecture}\label{mainconj}
	Asymptotics for Chebyshev polynomials in a complex domain $\cD$ can be given in terms of  the reproducing kernels of a suitable Hilbert space of analytic functions in $\cD$.
\end{conjecture}

The main motivation for this conjecture was given by the following two classical results:

\begin{itemize}
	\item[1.] Widom proved that, if the boundary $\pd\cD$ consists of finitely many (disjoint) smooth Jordan curves,  then asymptotics for the Chebyshev polynomials are represented in  terms of $H^\infty$-extremal functions in $\cD$ \cite[Theorem 8.3.]{Wid69}.
	\item[2.]  Garabedian, in his turn, already  in 1949, expressed certain extremal properties  of uniformly bounded in $\cD$ analytic functions, which form $H^\infty_{\cD}$, in terms of reproducing kernels \cite{Gar49}.
\end{itemize}

Actually, Garabedian called his paper ``Schwarz's lemma and the Szeg\"o kernel''. 
By the Ahlfors problem we mean the following task.
\begin{problem}\label{pA}
	Let $H^\infty_\cD$ be a collection of uniformly bounded analytic functions in a domain $\cD$. 
	Find
	\begin{equation*}
		\cA_\cD(z)=\sup\{|w'(z)|:\ \|w\|_{H^\infty_\cD}\le 1,\ w(z)=0\}, \quad z\in \cD.
	\end{equation*}
\end{problem}

Recall that for a compact $E=\pd\cD $ the proper defined value $A_{\cD}(\infty)$ is called the analytic capacity of $E$. Due to Garabedian,
$A_\cD(z)=k_{Sz}(z,z;\cD)$,  where $k_{Sz}(z,z_0;\cD)$ is the Szeg\"o reproducing kernel  in $\cD$.
This is, indeed, a generalization of the classical Schwarz lemma: if $\cD$ is the right half-plane, $\Re\l >0$, then
$$
\cA_{\cD}(\l_0)=\left(\frac{\l-\l_0}{\l+\bar \l_0}\right)'_{\l=\l_0}=\frac 1{\lambda_0+\bar \lambda_0}=k_{Sz}(\l_0,\l_0;\cD).
$$

Having in mind the
conformal invariance of our conjecture,  we will study simul\-taneously the following three problems, naturally related to the Ahlfors problem.

\begin{problem}\label{pJ} Let $E_J$ be a real compact consisting of $g+1$ non-degenerated intervals, 
	$E_J=[b_0,a_0]\setminus\cup_{j=1}^g(a_j,b_j)$. Let $\cP_n(E_J)$ be the collection of polynomials of degree $n$ bounded in absolute value by 1 on $E_J$. Define
	\begin{equation*}
		A_n(z;E_J)=\sup\{|P'(z)|:\ P\in \cP_n(E_J),\ P(z)=0\}, \quad z\in \bbC\setminus E_J.
	\end{equation*}
	Find asymptotics for $A_n(z;E_J)$ as $n\to\infty$.
\end{problem}

\begin{problem}\label{pT} Let $E_T$ be a system of arcs, 
	$E_T=\bbT\setminus\{e^{iz}: z\in\cup_{j=1}^g(a_j,b_j)\}$. 
	Let $\cP_n(E_T)$ be a collection of polynomials of degree $n$ bounded in absolute value by 1 on $E_T$. Define
	\begin{equation*}
		A_n(\z;E_T)=\sup\{|P'(\zeta)|:\ P\in \cP_n(E_T),\ P(\z)=0\}, \quad \z=e^{iz}\in \bbC\setminus E_T.
	\end{equation*}
	Find asymptotics for $A_n(\z;E_T)$ as $n\to\infty$.
\end{problem}

\begin{problem}\label{pS} Let
	$E_S=\bbR_+\setminus\cup_{j=1}^g(a_j,b_j)$. Let $\cE_\ell(E_S)$ be the collection of entire functions $F(z)$ of order $1/2$, of exponential type at most $\ell$ and bounded in absolute value by 1 on $E_S$, that is,
	$$
	|F(z)|\le C(\ell') e^{\ell'\sqrt{|z|}}, \ \forall\ \ell'>\ell , \quad |F(z)|\le 1\ \text{for} \ z\in E_S.
	$$ 
	Define
	\begin{equation*}
		A_\ell(z;E_S)=\sup\{|F'(z)|:\ F\in \cE_\ell(E_S),\ F(z)=0\}, \quad z\in \bbC\setminus E_S.
	\end{equation*}
	Find asymptotics for $A_\ell(z;E_S)$ as $\ell\to\infty$.
\end{problem}

We point out that, in fact, the mentioned Widom Theorem 8.3. \cite{Wid69} required extremal properties of \textit{multivalued} $H^\infty$-functions in $\cD$ (but with a single-valued absolute value). That is, to rewrite his result in terms of reproducing kernels one has to slightly generalize Problem \ref{pA}, for the exact setting see Problem \ref{pgA}, as well as to find a counterpart of Garabedian's theorem, which was done later by Abrahamse \cite{Ab79}. All these, including a proper definition of the Szeg\"o kernels are given in the preliminary section.
Now, let us formulate our solution of Problems \ref{pJ}, \ref{pT}, and \ref{pS} for simply connected domains (it seems, this  is already quite essential).

\begin{theorem}\label{thintro}
	Let $E_S=\bbR_+$, that is, $\cD=\{z=-\l^2:\ \Re\l>0 \}$. Then
	\begin{equation}\label{eq5}
		\Upsilon(\lambda):=\lim_{\ell\to\infty}e^{-\ell\Re\lambda}A_\ell(-\lambda^2;\bbR_+)2|\lambda|=
		\frac{1}{\l+\bar\l}\frac{2\sqrt{\l}\overline{\sqrt{\l}}}{(\sqrt{\l}+\overline{\sqrt{\l}})^2},
	\end{equation}
	as well as
	
	$$
	\Up(\l)=\lim_{n\to\infty}\left|\frac{\lambda-1}{\lambda+1}\right|^n A_n(z;[-2,2])\left|\frac{dz}{d\l}\right|=
	\lim_{n\to\infty}\left|\frac{\lambda-\l_0}{\lambda+\bar\l_0}\right|^n A_n(\zeta; E_T)\left|\frac{d\z}{d\l}\right|,
	$$
	where 
	$$
	z=2\frac{\l^2+1}{\l^2-1},\ \Re\l>0, \ \z=\frac{z-i y_0}{z+i y_0},\ \ y_0\in\bbR_+,
	$$
	and
	$$
	\lambda_0^2=\frac{iy_0-2}{i y_0+2}\ (\Re\l_0>0),\quad
	E_T=\left\{\z=\frac{z-i y_0}{z+iy_0}:\ z\in[-2,2]\right\}.
	$$
\end{theorem}

\begin{remark}
	Two remarks concerning Theorem \ref{thintro}:
	\begin{itemize}
		\item[(i)] \textit{Universality}. While the first exponential term in the asymptotics depends on the setting of the problem, 
		the second term $\Up(\l)$ is a conformally invariant value. Note that generally in Problems \ref{pJ}, \ref{pT}, \ref{pS} all three domains $\cD=\bbC\setminus E$, where $E=E_J, E_T, E_S$, (with a suitable choice of parameters) are conformally equivalent. In the same time they have certain features, in particular, they are related to different,  quite famous in the spectral theory, classes of operators: the so-called finite gap Jacobi, CMV matrices and 1-D Schr\"odinger operators. Recently we added to this family GMP matrices \cite{Y15}, which are related to approximation by rational functions with a prescribed system of poles. There is no doubt that after a suitable choice of the exponential factor the limit of the minimal deviation in such a rational approximation would lead to the same function $\Up(\l)$, see \cite{E16}, where such universality was demonstrated for the Chebyshev extremal problem. 
		\item[(ii)] \textit{Hilbert space structure}. $\log$-subharmonicity is a general property of upper envelopes of families of analytic functions \cite[Lecture 7]{Lev}, see also \cite{E16}. This explains why  the following matrix formed by partial derivatives should be nonnegative
		$$
		\begin{bmatrix}
		\Up(\l)&\pd \Up(\l)\\
		\bar\pd \Up(\l)&\bar\pd\pd \Up(\l)
		\end{bmatrix}\ge 0.
		$$
		But the limit value $\Up(\l)$ represents the diagonal of a certain \textit{reproducing kernel}
		$$
		k(\l,\l_0):=\frac{1}{\l+\bar\l_0}\frac{2\sqrt{\l}\overline{\sqrt{\l_0}}}{(\sqrt{\l}+\overline{\sqrt{\l_0}})^2},
		\quad\Re \l>0,\ \Re\l_0>0.
		$$
		Thus, for some reason we have an infinite number of inequalities of this sort: all of the following matrices are nonnegative
		\begin{equation}\label{hstr}
			[\bar\pd^m\pd^n \Up(\l)]_{n,m=0}^{N}\ge 0 \quad \text{for all}\ N\in\bbZ_+.
		\end{equation}
		We cannot comment appearance of this structure in the given context and leave this as an open problem.
		We point out that $k(\l,\l)$ is collinear to the Szeg\"o kernel only on the real axis, $\Up(\l)=\frac 1 2 k_{Sz}(\l,\l)$, $\l\in\bbR_+$.
		
	\end{itemize}
\end{remark}

Now we outline the structure of the paper and comment its other results.

The preliminary section contains statements, which are known at least on a folklore  level. To work with multivalued functions in a multi-connected domain $\cD$, we prefer to use a universal covering,
$\cD\simeq\bbC_+/\G$, where $\Gamma$ is a discrete subgroup of $SL_2(\bbR)$. We introduce multivalued \textit{complex} Green and Martin functions, define their characters and make a connection with conformal mappings on so-called comb-domains. 

The prime form is a standard object in the algebraic approach to the theory of reproducing kernels on Riemann surfaces \cite[Chap. II]{Fay}, \cite[Chap. IIIb]{MTata2}. The language of Hilbert spaces of automorphic forms 
$A_1^2(\Gamma,\alpha)$  is probably much easier for specialists in analysis. In this way we define the Szeg\"o kernel $k^\alpha_{Sz}(\l,\l_0)$, Definition \ref{defszz}. In fact, it is not that much important whether one works with Hardy spaces of character automorphic functions or forms.  But the relation between the character $\beta\in\G^*$ in the character automorphic Ahlfors Problem \ref{pgA} and its solution, Theorem \ref{th28}, looks particularly simple in the second version, $\alpha^2=\beta$. Note that the extremal character $\alpha$ here is defined up to a half period $\fj$, $\fj^2=1_{\G^*}$.  We demonstrate that Garabedian's case of the trivial character, $\beta=1_{\G^*}$, in which the extremal character does not depend on $\lambda_0$,  Theorem \ref{th211}, is an exception. In fact, in general the half period varies with $\lambda_0$, $\fj=\fj(\l_0)$.

An  interrelation of the Ahlfors and Abel-Jacobi 
inversion problems was noted in \cite{Gar49}. While the theory for orthogonal polynomials \cite{AkhCongr}, or the related to it spectral theory for finite gap Jacobi matrices leads to the classical  Abel-Jacobi inversion problem  in the proper sense \cite{MTata2}, the character automorphic $H^\infty$-extremal problem requires  a certain modification, Proposition \ref{propgj1}. We clarify this in the last preliminary subsection \ref{ss23}.

Our asymptotic relation for $A_n(z;E_J)$ for a real $z$, Theorem \ref{thrmain}, is  the Widom 
Theorem 11.5 \cite{Wid69} with  a specific weight, see also  \cite{ChrSiZi15}. But asymptotics $A_n(\z;E_T)$, $\z\in \bbT$, would require  a certain varying  weight in such a  reduction.  Instead, in this symmetric case we solve all  Problems \ref{pJ}, \ref{pT}, and \ref{pS}  in a unified way: using Chebyshev alternation theorem we represent extremal functions in terms of correspon\-ding comb functions, after that we use a simple relation \eqref{eq38} between Martin/Green functions of the given domains and its $\ell/n$-regular restriction. Thus, in this case Conjecture \ref{mainconj} gets its confirmation with the Szeg\"o character automorphic reproducing kernel.

Finally, using the Kolmogorov theorem, we show how to move $z_0$ in the complex plane. We get a reduction of the extremal problem to a generalization of the modified Abel-Jacobi problem, see Problem \ref{gaji}.
As it was already  mentioned, the value $\Up(z_0,\beta)$, which is responsible for the asymptotics, is universal, but  \textit{is  not collinear to the Szeg\"o reproducing kernel} as soon as $\Im z_0\not=0$, Theorem \ref{cth42},
see also Remark \ref{30rem2}.

\section{Preliminaries}
\subsection{Comb-domains and elements of potential theory}

The following comb-domains are standard objects in the spectral theory of reflectionless operators
\cite{MarOst75}, see also \cite{ErY12, DY16}.
Let, see Fig. \ref{comb},
\begin{align*}
	\Pi_J= &\{\vt=\xi+i\eta: 0<\xi<\pi,\ \eta>0\}\setminus \cup_{j=1}^g\{\vt=\omega_j+i\eta,\eta\in(0,h_j]\},
	\\
	\Pi_T=&\bbC_+\setminus\cup_{j=0}^g\cup_{m\in\bbZ}\{\vt=\omega_j+2\pi m+i\eta, \eta\in (0,h_j]\},
	\\
	\Pi_S=&\{\vt=\xi+i\eta: \xi>0,\ \eta>0\}\setminus \cup_{j=1}^g\{\vt=\omega_j+i\eta,\eta\in(0,h_j]\}.
\end{align*}
In the first case $\omega_j\in(0,\pi)$ in the second one $\omega_0=0$ and $\omega_j\in(0,2\pi)$ for $j=1,..., g$.
We map conformally $\bbC_+$ on one of the corresponding combs making normalizations
\begin{align}
	\tau_J(b_0)=0,\quad \tau_J(a_0)=\pi, \quad \tau_J(\infty)=\infty, \label{eq24}
	\\
	\tau_T(iy)\simeq iy, \ y\to\infty,\quad \tau_T(0)=0,\label{eq25}
	\\
	\tau_S(-x)\simeq i\sqrt{x},\ x\to\infty, \quad \tau_S(0)=0.\label{eq26}
\end{align}
Note that in the second case we get automatically $\tau_T(z+2\pi)=\tau(z)+2\pi$ \cite{DY16}, that is, $e^{i\tau_T}$ is well defined as a function of $\z=e^{iz}\in\bbD$. In the first and third case we get the gaps $(a_j,b_j)$ as preimages of the corresponding vertical slits, $j=1,...,g$. In the same way, in the second case we get a system of arcs $\{\z=e^{iz}: z\in(a_j,b_j)\}_{j=0}^g$, which form the complement of $E_T$.

\begin{figure}[htbp] 
	\begin{center}
		\includegraphics[scale=0.6]
		{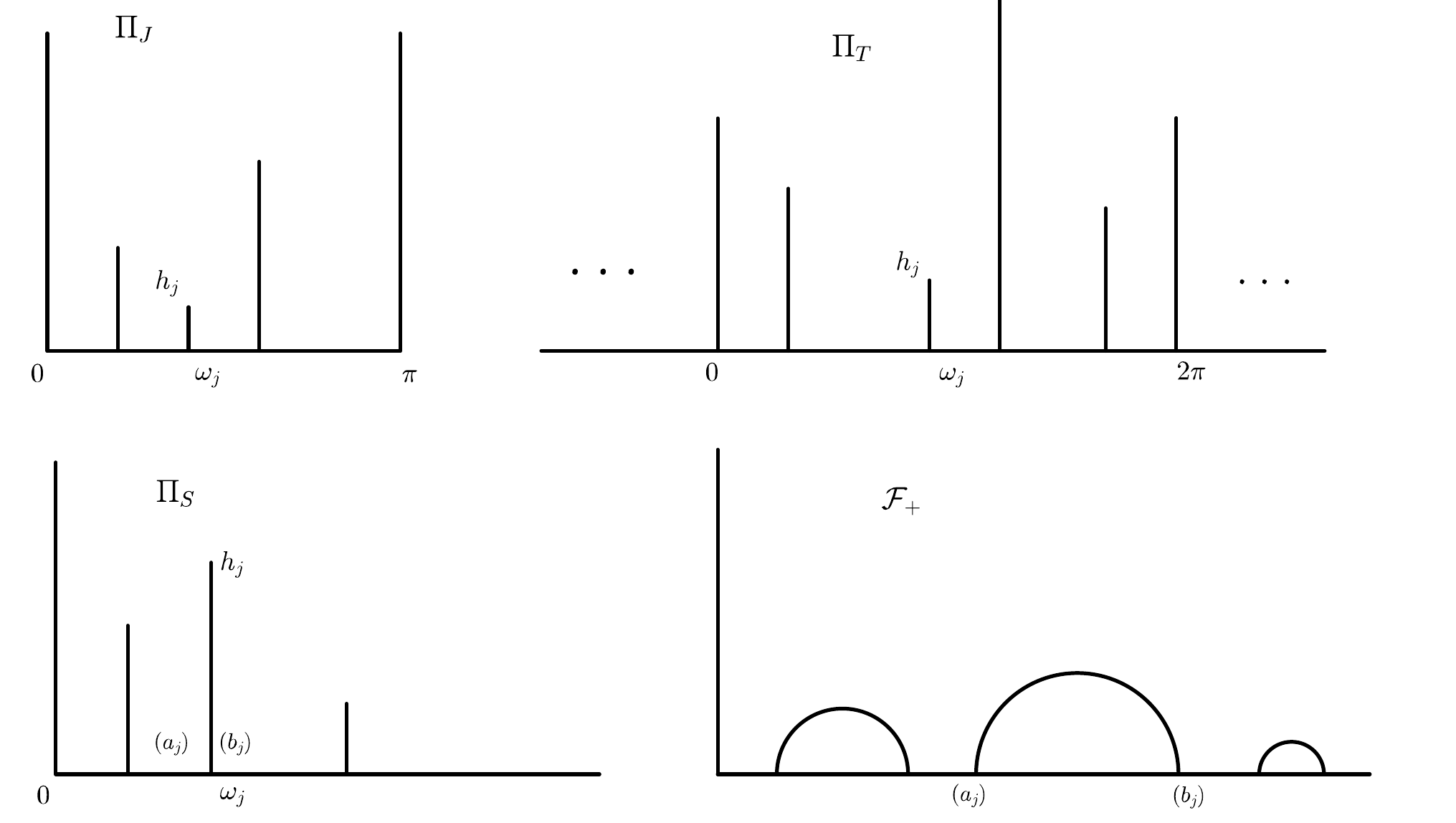}
		\caption{Comb-domains and the uniformization plane}
		\label{comb}
	\end{center}
\end{figure}

If $\tau$ denotes one of the maps \eqref{eq24}--\eqref{eq26},
then $\Im \tau(z)$ can be extended through the gaps to a single valued harmonic function in $\cD$. Moreover,
\begin{equation}\label{eq270}
	\Im \tau_J(z)=G(z,\infty),\quad \Im\tau_T(z)={G(e^{iz},0)+G(e^{iz},\infty)},\quad \Im\tau_S(z)=M(z),
\end{equation}
where $G(z,z_0)$ denotes the Green function w.r.t. $z_0$ in the corresponding domain and $M(z)$ stands for the Martin function w.r.t. infinity. The function $e^{i\tau(z)}$ can be extended to $\cD$ by the symmetry principle as a \textit{multivalued} function. Such functions become single-valued on a universal covering.

Let us fix $\cD$ in the form $\bbC\setminus E_S$. Recall that for an arbitrary $\bar\bbC\setminus E_T$ or 
$\bar\bbC\setminus E_J$ we always can find a suitable conformally equivalent domain $\cD$ of the above form. A universal covering $z=z(\lambda)$, $\l\in\bbC_+$, also corresponds to a conformal mapping. For a given $E_S$ there exists a system of half discs $\bbD_j^+$ such that the conformal mapping 
$$
\bbC_+\to\cF_+=\{\lambda=\xi+i\eta: \xi>0,\eta>0\}\setminus \cup_{j=1}^g\bbD^+_j,\quad \lambda(-x)\simeq i\sqrt{x}, \
x\to\infty,
$$
transforms the negative half axis into the imaginary half axis in the $\lambda$-plane and the gaps $(a_j,b_j)$ into the boundary of the half discs $\pd \bbD_j^+$, see Fig. \ref{comb}. The inverse map $z=z(\lambda)$ can be extended by a system of reflections to the whole upper half-plane. Indeed, let $\gamma_j$, acting in the $\lambda$-plane corresponds to the double reflection w.r.t. the negative half axis and the gap $(a_j,b_j)$ in the $z$-plane. This is a  linear fractional transform, which maps $\pd \bbD_j^+$ on
$-\overline{\pd \bbD_j^+}$ and we have $z(\g_j(\l))=z(\l)$. In this case, the system $\{\g_j\}_{j=1}^g$ represents a generator of the Fuchsian group $\Gamma$ and $\cF=\cF_+\cup\overline{-\cF_+}\cup{i\bbR_+}$ is a fundamental  domain for the action of $\Gamma$ on $\bbC_+$ such that $ \bbC_+/\Gamma\simeq \bbC\setminus E_S$
(respectively to $\bar\bbC\setminus E_J$ or $\bar\bbC\setminus E_T$).

Conformal mappings on comb-domains are partial cases of Schwarz-Christoffel transformations. Due to the classical formula, say,
\begin{equation*}
	\tau_J(z)=i\int_{b_0}^z\prod_{j=1}^g\frac{z-c_j}{\sqrt{(z-a_j)(z-b_j)}}\frac{dz}{\sqrt{(z-a_0)(z-b_0)}},
\end{equation*}
where $c_j\in(a_j,b_j)$ corresponds to the top $\omega_j+ih_j$ of the slit.

\begin{figure}[htbp] 
	\begin{center}
		\includegraphics[scale=0.17]
		{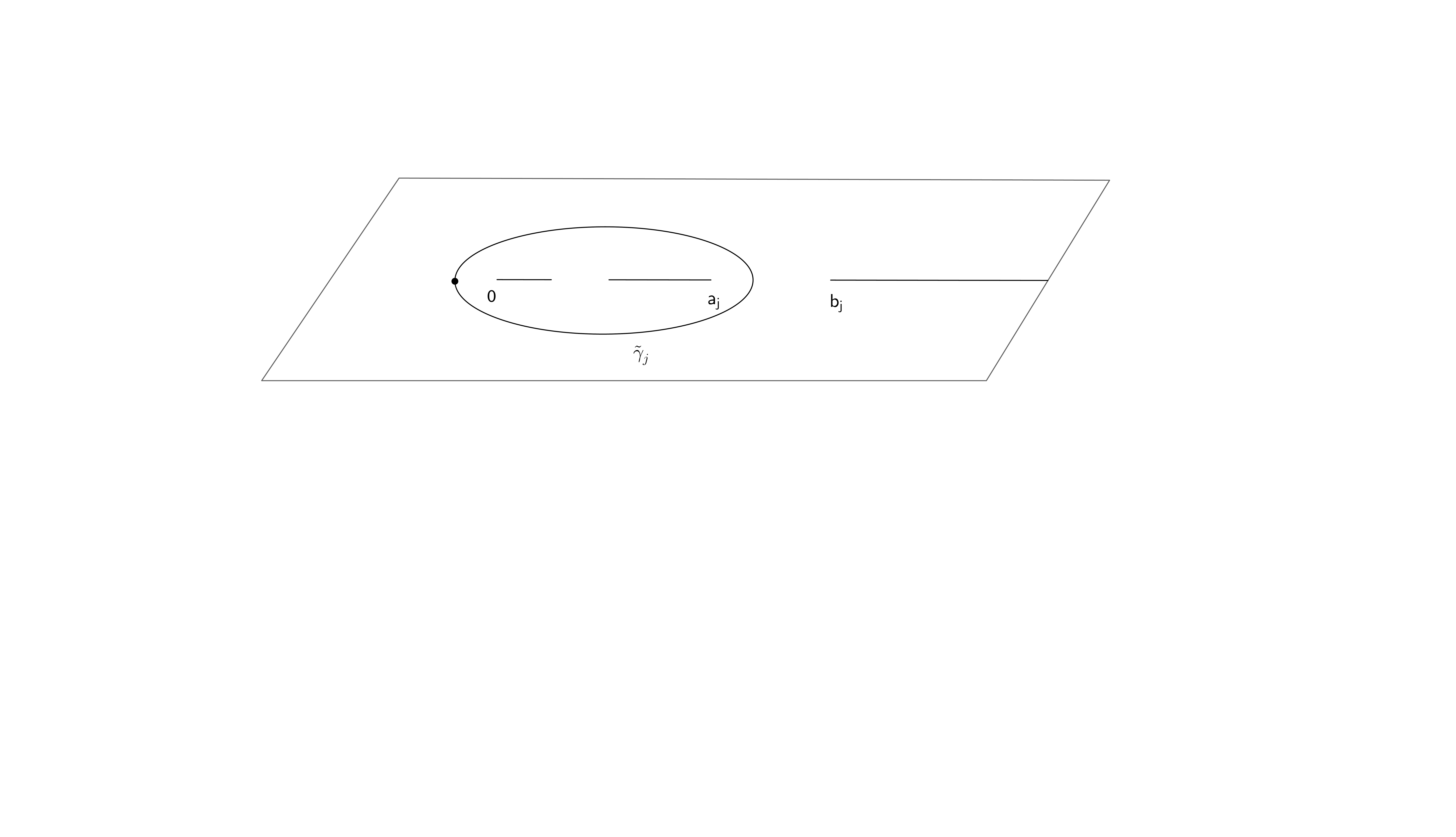}
		\caption{A generator of the fundamental group}
		\label{fund}
	\end{center}
\end{figure}

Now we extend $e^{i\tau_S(z)}$ along the generator $\tilde \gamma_j$ of the fundamental group in $\cD$, see Fig. \ref{fund}, which corresponds to the action of $\gamma_j$ on the universal covering. As result we obtain
$$
e^{i\tau_S(z(\gamma_j(\lambda))}=\alpha_S(\gamma_j)e^{i\tau_S(z(\l))}, \quad \alpha_S(\gamma_j)=e^{2\omega_j i}.
$$
This system of multipliers given on the generators forms an element $\alpha_S(\gamma)$, $\gamma\in\Gamma$, of  the group of \textit{characters} $\Gamma^*$. Similar relations generate characters $\alpha_J$ and $\alpha_T$. 

The function $e^{i\tau_J(z(\lambda))}$ is called the complex Green function of the group $\Gamma$, and can be represented as the Blaschke product $b_\infty(\lambda)=e^{i\tau_J(z(\lambda))}$ along the trajectory $z^{-1}(\infty)$. Generally, for $z_0=z(\lambda_0)$ we have
\begin{equation*}
	b_{z_0}(\lambda)=\prod_{\gamma\in\Gamma}
	\frac{\lambda-\gamma(\lambda_0)}{\lambda-\overline{\g(\lambda_0)}}e^{i\psi_{\gamma}},\quad e^{i\psi_\g}:=
	\frac{i-\overline{\g(\lambda_0)}}{i-\gamma(\lambda_0)}
	\left|
	\frac{i-\gamma(\lambda_0)}{i-\overline{\g(\lambda_0)}}\right|.
\end{equation*}
We point out that
\begin{equation*}
	-\log|b_{z_0}(\lambda)|=G(z(\lambda),z_0).
\end{equation*}
The character generated by $b_{z_0}$ is denoted by $\mu_{z_0}$, in particular, $\mu_\infty=\alpha_J$.

In the domain $\cD=\bar\bbC\setminus E_T$ we have, see \eqref{eq270},
$$
G(\zeta(\lambda),0)=\frac{-\log|\zeta(\lambda)|+\Im\tau_T(\z(\lambda))}{2},\quad
\z(\lambda)=\frac{b_0(\lambda)}{b_\infty(\lambda)}.
$$
That is,
\begin{equation*}
	b_0(\lambda)=\sqrt{\z(\lambda)e^{i\tau_T(\z(\lambda))}}\quad \text{and}\quad \mu_0^2=\alpha_T,\ 
	\mu_0(\gamma_j)=\mu_\infty(\gamma_j)=e^{i  \omega_j}.
\end{equation*}

Finally, we have to mention the well known relation between $\omega_j$'s and the harmonic measures.
Let $\omega(z,F)$ be the harmonic measure of the set $F\subset E$ in the domain $\cD=\bbC\setminus E$ w.r.t. $z\in \cD$. In $\bar\bbC\setminus E_J$ we have
\begin{equation}\label{212}
	\omega_k=\pi \omega(\infty, E_J^k), \ E^k_J=E_J\cap[b_0,a_k], 
\end{equation}
and 
\begin{equation}\label{213}
	\omega_k=2\pi\omega(\infty, E_T^k)=2\pi\omega(0, E_T^k), \ E^k_T=E_T\cap\{e^{iz}:\ z\in [0,a_k]\}
\end{equation}
in $\bar\bbC\setminus E_T$.

Relations of the Martin function in $\cD$ with the complex Martin function of the group $\Gamma$ we discuss in the next subsection.

\subsection{Szeg\"o kernel on the universal covering and  Ahlfors problem}
Let $H^p$ be the standard Hardy space in $\bbC_+$,
$$
\|f\|^p=\frac 1{2\pi}\int_{\bbR}|f(\xi)|^pd\xi, \quad f\in H^p,
$$
with a proper modification for $p=\infty$.
For a fixed character $\alpha\in\Gamma^*$ we introduce 
$$
H^p(\alpha)=\{f(\lambda):\ f\in H^p, \quad f(\gamma(\lambda))=\alpha(\gamma)f(\lambda)\}.
$$

\begin{lemma}
	If  $H^2(\alpha)\not=\{0\}$ for some $\alpha$, then
	\begin{equation}\label{214}
		\sum_{\gamma\in\Gamma}\gamma'(\xi)<\infty
	\end{equation}
	for almost all $\xi\in\bbR$.
\end{lemma}

\begin{proof}
	In fact, if there is   a measurable fundamental set $\bbE=\pd\cF\cap \bbR$ for the action of $\Gamma$ on $\bbR$ and a positive automorphic function $f\in L^1$, $f\circ\gamma=f$, non-vanishing almost everywhere, then
	$$
	\frac 1{2\pi}\int_\bbR f(\xi) d\xi=\frac 1{2\pi}\int_\bbE\{\sum_{\gamma\in\Gamma}\gamma'(\xi) \} f(\xi) d\xi<\infty,
	$$
	and we get \eqref{214}.
\end{proof}

For a reason that will be clear in a moment we consider $\Gamma$ as a subgroup of $SL_2(\bbR)$
(see Remark \ref{rem23}).
The condition \eqref{214} guarantees that the following series converges in the upper half-plane
\begin{equation}\label{eq215}
	\sum_{\g\in\G}\Im\g(\l)=\sum_{\g\in\G}\frac{\Im\l}{|\g_{21}\l+\g_{22}|^2}, \quad
	\gamma=\begin{bmatrix}\g_{11}&\g_{12}\\ \g_{21}&\g_{22}
	\end{bmatrix},\ \g_{ij}\in\bbR,\  \det \gamma=1.
\end{equation}
In other words the mass-point measure supported on the trajectory of infinity $\xi_\g=-\frac{\g_{22}}{\g_{21}}=\gamma^{-1}(\infty)$  with masses $\sigma_\g=\frac{1}{\g_{21}^2}$ 
satisfies the condition
$$
\sum_{\g\in\Gamma, \g\not=1_{\G}}\frac{\sigma_{\g}}{1+\xi_\g^2}=
\sum_{\g\in\Gamma, \g\not=1_{\G}}\frac{1}{\g_{21}^2+\g_{22}^2}<\infty.
$$
Thus, the corresponding function
$$
\cM(\l)=\l+\sum_{\g\in\Gamma, \g\not=1_{\G}}\frac{1+\l\xi_\g}{\xi_\g-\l}\frac{\sigma_\g}{1+\xi_\g^2}
$$
is well defined, has positive imaginary part  given by \eqref{eq215}, moreover, $\Im \cM(\gamma(\lambda))=\Im \cM(\lambda)$. That is, $e^{i\cM(\lambda)}$ is the complex Martin function of the group $\Gamma$ w.r.t. infinity, $\Im \cM(\l)=M(z(\l))$ and for  $z(\lambda)\simeq \lambda^2$, $\lambda=i\eta$, $\eta\to\infty$, we have
\begin{equation}\label{216}
	\cM(\lambda)-\cM(0)=\tau_S(z(\lambda))=-i\int_0^{z(\l)}\prod_{j=1}^g\frac{z-c_j}{\sqrt{(z-a_j)(z-b_j)}}\frac{dz}{2\sqrt{-z}}.
\end{equation}

The Blaschke product along the trajectories corresponding to the critical points
$$
\cW(\lambda)=\prod_{j=1}^g b_{c_j}(\l)
$$
is called the Widom function. Note that convergence of this product is called Widom condition for the given group (domain). Such Fuchsian groups were studied by Pommerenke \cite{Pom76}. For a complement of a system of intervals the Widom condition holds obviously.

\begin{theorem}[Pommerenke] The function $\cM'(\lambda)$ is holomorphic in the upper half-plane with zeros
	at $\{z^{-1}(c_j)\}_{j=1}^g$, see \eqref{216}. Moreover, it is of bounded chara\-cteristic in the upper half-plane
	such that
	\begin{equation}\label{217}
		\frac{d\cM(\l)}{d\l}=
		\sum_{\gamma\in\Gamma}\frac{1}{(\g_{21}\lambda+\g_{22})^2}=
		\sum_{\gamma\in\Gamma}\gamma'(\lambda)
		=\cW(\l) \vpi(\l)^2,
	\end{equation}
	where $\vpi$ is an outer function. Respectively, for its boundary values  on the real axis, one has
	\begin{equation}\label{218}
		\frac{d\cM(\l)}{d\l}= |\vpi(\l)|^2=\sum_{\gamma\in\Gamma}\frac{1}{|\g_{21}\lambda+\g_{22}|^2}
		\ge 1, \quad \cW(\lambda)\overline{\vpi(\l)}=\vpi(\l).
	\end{equation}
\end{theorem}

For an analytic function in $\bbC_+$, $\gamma\in\Gamma$,  we write
\begin{equation}\label{219}
	f|[\gamma]=\frac{f(\gamma(\lambda))}{(\gamma_{21}\lambda+\gamma_{22})},\quad \gamma=
	\begin{bmatrix}\gamma_{11}&\gamma_{12}\\
		\gamma_{21}&\gamma_{22}
	\end{bmatrix}.
\end{equation}
Let $\alpha_{\cW}\in\G^*$ be the character of the Widom function $\cW$. Then, see \eqref{217},
\begin{equation}\label{220}
	(\cM(\gamma(\l)))'=\cM'(\l)\Rightarrow \vpi|[\g]=\nu(\g)\vpi,
\end{equation}
where $\nu$ denotes a certain fixed root of the character $\alpha_{\cW}^{-1}$, $\nu^2=\alpha_{\cW}^{-1}$. 

\begin{remark}\label{rem23}
	Generally, a square root of a character is defined up to a half period, that is, up to a character $\fj\in\Gamma^*$ such that $\fj^2=1_{\G^*}$. 
	Note that  in the definition \eqref{219} it is essential that $\Gamma$ is considered as a subgroup of $SL_2(\bbR)$, but not as a Fuchsian group, where $\pm\gamma$ generates the same transform.
	In fact, such passage is also  defined up to a choice of a half period. Indeed,
	any of the groups
	\begin{equation}\label{eqgj}
		\G_\fj=\{\fj(\g)\g:\ \g\in\G\}\subset SL_2(\bbR),\quad \fj\in\G^*,\ \fj^2=1_{\G^*},
	\end{equation}
	generates  the same group of linear fractional transforms.
\end{remark}

Recall  that a function of bounded characteristic in $\bbC_+$ is of Smirnov class, or of Nevanlinna class $\cN_+$ in another terminology \cite[Chap. II Sect. 5]{Gar07}, if it can be represented as a ratio of two functions from $H^\infty$ with an \textit{outer} denominator. Note that functions of this class obey a maximum principle of a high generality. For instance, if $f\in \cN_+$ and its boundary values satisfy $f\in L^2$, then $f\in H^2$.

\begin{definition}
	For $\alpha\in\Gamma^*$  the space $A_{1}^{2}(\Gamma,\alpha)$ is formed by those analytic functions $f$ in $\bbC_+$ that satisfy the following three conditions
	\begin{itemize}
		\item[(i)] $f$ is of Smirnov class,
		\item[(ii)] $f|[\gamma]=\alpha(\gamma) f \ \forall \gamma\in\Gamma$,
		\item[(iii)] $\frac 1 {2\pi}\int_{\bbE}|f(\lambda)|^{2}d\lambda<\infty$.
	\end{itemize}
\end{definition}

\begin{proposition}\label{prop25}
	The following map $f\mapsto \vpi f$ sets a unitary correspondence between $H^2(\alpha)$ and $A_1^2(\G,\nu\alpha)$.
\end{proposition}

\begin{proof}
	Let $f\in H^2(\alpha)$. Then
	\begin{equation}\label{eqhsp}
		\|f\|^2=\frac 1{2\pi}\int_{\bbE}|f(\l)|^2\sum_{\g\in\G}\g'(\l)d\l=\frac 1{2\pi}\int_\bbE |f(\l)\vpi(\l)|^2d\l.
	\end{equation}
	Since $f\in H^2$ and $\vpi$ is outer  $f\vpi\in\cN_+$. The property (ii) follows from \eqref{220}.
	Conversely, if $g\in A_1^2(\G,\nu\alpha)$, then $f=g/\vpi$ is in the standard $L^2$ and of Smirnov class. Thus, it belongs to $H^2$. The ratio of two forms, see (ii), generates a function with character $\alpha$.
\end{proof}

The point evaluation functional is bounded in $H^2$. By $k^\alpha_{\l_0}(\l)=k^\alpha(\l,\l_0)$ we denote the reproducing kernel in $H^2(\alpha)$, $\langle f,k^\alpha_{\l_0} \rangle=f(\l_0)$ for all $f\in H^2(\alpha)$.

\begin{definition}\label{defszz}
	The reproducing kernel of the space $A_{1}^2(\G,\alpha)$ we call the Szeg\"o kernel (corresponding to the given group $\G$ and its character $\alpha)$,
	\begin{equation*}
		k_{Sz}(\lambda,\l_0;\Gamma,\alpha)=k^\alpha_{Sz}(\lambda,\l_0)=k^{\alpha\nu^{-1}}(\l,\l_0)\vpi(\l)\overline{\vpi(\l_0)}.
	\end{equation*}
	
\end{definition}

We slightly generalize the Ahlfors problem \ref{pA}.

\begin{problem}\label{pgA}
	For $\lambda_0\in\bbC_+$ find
	$$
	\cA(\lambda_0,\beta)=\sup\{|w'(\lambda_0)|: w\in H^\infty(\beta),\ \|w\|\le 1, \ w(\lambda_0)=0\}.
	$$
\end{problem}

\begin{theorem}\label{th28}
	A solution of Problem \ref{pgA} is given in terms of the Szeg\"o kernels
	\begin{equation}\label{222}
		\cA(\l_0,\beta)=\min_{\alpha^2=\beta}k^{\alpha}_{Sz}(\lambda_0,\l_0)=k_{Sz}^{\alpha(\l_0)}(\l_0,\l_0).
	\end{equation}
	If $\beta=1_{\G^*}$ the above minimum is assumed on the same half period $\fj\in\Gamma^*$ for all $\lambda_0\in\bbC_+$, that is, see \eqref{eqgj}, the Garabedian formula holds
	\begin{equation}\label{2222}
		\cA(\l_0,1_{\G^*})=\cA(\l_0,1_{\G_\fj^*})=
		k_{Sz}(\lambda_0,\l_0;\G_\fj, 1_{\G_\fj^*}).
	\end{equation}
	Generally,  the extremal character $\alpha(\l_0)$, $\alpha(\l_0)^2=\beta$, depends on $\l_0$.
\end{theorem}

\begin{definition}\label{defj}
	In what follows the choice of the group $\Gamma\subset SL_2(\bbR)$ is assumed to match with the extremal half period in the Ahlfors problem, see Remark \ref{rem23}. That is, cf. \eqref{2222},
	$$
	\cA(\l_0,1_{\G^*})=
	k_{Sz}(\lambda_0,\l_0;\G, 1_{\G^*})=k_{Sz}(\l_0,\l_0).
	$$
\end{definition}

As it was mentioned in Garabedian's original paper \cite{Gar49}, Problem \ref{pA} is the easiest  version  of Nevanlinna-Pick kind problems for multi connected domains (Rie\-mann surfaces). They were studied later by Abrahamse \cite{Ab79} and in many consequent papers, see e.g. \cite{BV96,KY97}.
The statement below is an easy consequence of Abrahamse's theorem. Note that a similar statement can be founded in an essentially much more general situation in \cite{VY14}.

\begin{proposition}\label{lprop22}
	\begin{equation}\label{eq225}
		\sup_{w\in H^\infty(\beta),\ \|w\|\le 1} |w(\l_0)|^2=\inf_{\alpha\in \Gamma^*}\frac{k_{Sz}^{\alpha\beta}(\l_0,\l_0)}
		{k_{Sz}^\alpha(\l_0,\l_0)}=\frac{k_{Sz}^{\alpha_0\beta}(\l_0,\l_0)}{k_{Sz}^{\alpha_0}(\l_0,\l_0)},
		\ \ \alpha_0=\alpha_0(\l_0).
	\end{equation}
	Moreover, an extremal function $w_{\l_0,\beta}(\l)$ is a Blaschke product and
	\begin{equation}\label{eq226}
		w_{\l_0,\beta}(\l)\overline{w_{\l_0,\beta}(\l_0)}=\frac{k_{Sz}^{\alpha_0\beta}(\l,\l_0)}{k_{Sz}^{\alpha_0}(\l,\l_0)}.
	\end{equation}
\end{proposition}
Note that one side of the statement deals with the trivial estimation
\begin{align}
	|w(\lambda_0)k^{\alpha}_{Sz}(\l_0,\l_0)|^2=&|\langle wk^{\alpha}_{Sz,\l_0},k^{\alpha\beta}_{Sz,\l_0} \rangle|^2
	\nonumber \\ \le &
	\|k^{\alpha}_{Sz,\l_0}\|^2\|k^{\alpha\beta}_{Sz,\l_0}\|^2=
	k^{\alpha}_{Sz,\l_0}(\l_0) k^{\alpha\beta}_{Sz,\l_0}(\l_0).\label{mineq}
\end{align}
Also, due to the nature of formulas \eqref{eq225}, \eqref{eq226}, it does not matter to use Szeg\"o kernels, or the reproducing kernels of character automorphic Hardy spaces.

\begin{proof}[Proof of \eqref{222} in Theorem \ref{th28}] Let $v$ be an extremal  function for Problem \ref{pgA}. Then
	$v=b_{z_0}w$, where $z_0=z(\l_0)$ and $w=w_{\l_0,\beta\mu_{z_0}^{-1}}$ is an extremal function from
	\eqref{eq226}.
	
	Now we use duality \cite{Yu01} $L^2_{d\lambda|\bbE}\ominus A_1^2(\Gamma,\alpha)=\overline{A^2_1(\Gamma,\alpha^{-1})}$, $\forall \alpha\in\Gamma^*$, due to which
	$$
	b_{z_0}(\lambda)\frac{\overline{k_{Sz}^{\alpha}(\lambda,\lambda_0)}}{\|k^\alpha_{Sz,\lambda_0}\|}=
	\frac{{k_{Sz}^{\alpha^{-1}\mu_{z_0}}(\lambda,\lambda_0)}}{\|k^{\alpha^{-1}\mu_{z_0}}_{Sz,\lambda_0}\|},
	\quad \lambda\in\bbR,
	$$
	and
	\begin{equation}\label{eq227}
		{k_{Sz}^{\alpha^{-1}\mu_{z_0}}(\lambda_0,\lambda_0)}{k_{Sz}^{\alpha}(\lambda_0,\lambda_0)}=
		|b'_{z_0}(\lambda_0)|^2.
	\end{equation}
	Using $|w(\l)|=1$ on $\bbR$ we get, simultaneously \eqref{eq226} and the dual representation
	\begin{equation}\label{1eq227}
		w(\l)\overline{w(\lambda_0)}=\frac{k_{Sz}^{\alpha_0\beta}(\l,\l_0)}{k^{\alpha_0\mu_{z_0}}_{Sz}(\l,\l_0)}
		=\frac{k_{Sz}^{\alpha_0^{-1}}(\l,\l_0)}{k^{\alpha^{-1}_0\beta^{-1}\mu_{z_0}}_{Sz}(\l,\l_0)}.
	\end{equation}
	Generically, a solution of the problem is a unique Blaschke product of $g$ complex Green functions, see 
	Theorem \ref{th211}  and \cite{Ab79, KY97}, that is, $\alpha_0=\alpha_{0}^{-1}\beta^{-1}$ and for this character
	$$
	|v'(\l_0)|^2=|b'_{z_0}(\l_0)|^2|w(\lambda_0)|^2=
	|b'_{z_0}(\l_0)|^2\frac{k_{Sz}^{\alpha_0\beta}(\l_0,\l_0)}{k^{\alpha_0\mu_{z_0}}_{Sz}(\l_0,\l_0)}.
	$$
	By \eqref{eq227}, we have
	$$
	|b'_{z_0}(\l_0)|^2\frac{k_{Sz}^{\alpha^{-1}\beta}(\l_0,\l_0)}{k^{\alpha^{-1}\mu_{z_0}}_{Sz}(\l_0,\l_0)}=|b'_{z_0}(\l_0)|^2\frac{k_{Sz}^{\alpha}(\l_0,\l_0)}{k^{\alpha^{-1}\mu_{z_0}}_{Sz}(\l_0,\l_0)}
	={k_{Sz}^{\alpha}(\l_0,\l_0)}^2
	$$
	for an arbitrary $\alpha$ such that $\alpha^2=\beta$. Due to \eqref{eq225}, we obtain \eqref{222} with $\alpha(\l_0)=\alpha_0^{-1}$. In degenerated cases the formula holds by continuity, although the choice for $\alpha_0$ is not unique.
\end{proof}

Now we will essentially specify the second statement of Theorem \ref{th28} for finitely connected Denjoy domains.
Note that the general formula for the analytic capacity in  Denjoy domains is a well known result of Pommerenke \cite{Pom1}, see also \cite[\S 8.8]{Sim15}.

\begin{theorem}\label{th211} Let $\cD=\bbC\setminus E_S$ and
	\begin{equation}\label{eq228}
		\Omega(z)=\frac{1}{\sqrt{-z}}\prod_{j=1}^g\sqrt{\frac{z-a_j}{z-b_j}}.
	\end{equation}
	Then
	\begin{equation}\label{eq229}
		k_{Sz}(\lambda_0,\lambda_0)=
		\frac{\Im \Omega(z_0)}{2 \Im z_0|\Omega(z_0)|} \left|\frac{dz}{d\lambda}(\lambda_0)\right|,
		\quad z_0=z(\lambda_0).
	\end{equation}
	Respectively, the extremal Ahlfors function of Problem \ref{pA} is of the form
	\begin{equation}\label{eq230}
		w_{z_0, \cD}(z)=\frac{z-z_0}{z-\bar z_0}\frac{\Omega(z)-\Omega(\bar z_0)}{\Omega(z)+\Omega(z_0)}.
	\end{equation}
\end{theorem}

First we prove the following lemma.

\begin{lemma}\label{l212}
	Let $H^2_{\Omega}$ be the space of Smirnov class functions $F(z)$ in $\cD$ with the scalar product
	$$
	\|F\|^2_{\Omega}=\frac{1}{\pi}\int_E\frac{|F(x+ i0)|^2+|F(x-i0)|^2}{2}\Im\Omega(x) dx=
	\frac 1{2\pi i}\oint_E |F(x)|^2\Omega(x)dx.
	$$
	Then the reproducing kernel of this space is of the form
	\begin{equation}\label{repom}
		K_{\Omega}(z,z_0)=\frac{-\Omega^{-1}(z)+\Omega^{-1}(\bar z_0)}{{2}(z-\bar z_0)}.
	\end{equation}
\end{lemma}

\begin{proof}
	Note that $\overline{\Omega(x)}=-\Omega(x)$ if $x\in E$ and 
	$\overline{\Omega(\bar z)}=\Omega(z)$ for $z\in \cD$. For $F\in H^2_{\Omega}$ we can apply the Cauchy theorem
	$$
	\langle F,K_{\Omega}(\cdot,z_0) \rangle=\frac 1{2\pi i}\oint_E  \frac{\Omega^{-1}(x)+\Omega^{-1}(z_0)}{2(x-z_0)}F(x)\Omega(x)dx= F(z_0).
	$$
\end{proof}

\begin{proof}[Proof of Theorem \ref{th211}] According to Proposition \ref{prop25} the space $A_1^2(\Gamma)$ can be interpreted as the collection of functions $f\vpi$, where $f\in H^2(\nu^{-1})$. In its turn, to
	$f\in H^2(\nu^{-1})$ we associate a multivalued function $F$ in $\cD$ such that $F(z(\lambda))=f(\lambda)$. Note that $|F(z)|$ is single valued in $\cD$, and the scalar product, by \eqref{218} and \eqref{eqhsp}, has the form
	$$
	\|f\|^2_{H^2(\nu^{-1})}=\frac 1 {\pi} \int_E\frac{|F(x+i0)|^2+|F(x-i0)|^2}{2}\prod_{j=1}^g\frac{x-c_j}{\sqrt{(x-a_j)(x-b_j)}}\frac{dx}{2\sqrt{x}}.
	$$
	
	Let $\phi_{\Omega}$ be the multivalued outer function in $\cD$ given by
	\begin{equation}\label{eq232}
		\phi_\Omega^{-2}=\frac 1 {2\cW}\prod_{j=1}^g\frac{z-c_j}{z-a_j}.
	\end{equation}
	Note that
	$$
	|\phi_\Omega|^{-2}=\frac{1}{2}\prod_{j=1}^g\frac{x-c_j}{x-a_j},\ x\in E,\ \text{and}\ \phi_{\Omega}\circ\gamma=\nu^{-1}(\g)
	\phi_{\Omega}
	$$
	(at this point we fix the half period for $\Gamma$, see Definition \ref{defj}, so that the function
	$\phi_{\Omega}$ and the form $\vpi$ generate the  mutually  inverse characters).
	Since
	$$
	\|f\|^2_{H^2(\nu^{-1})}=\frac 1 {\pi} \int_E\frac{|(F/\phi_{\Omega})(x+i0)|^2+|(F/\phi_{\Omega})(x-i0)|^2}{2}
	\Im\Omega(x) {dx},
	$$
	for $g\in A_1^2(\G)$ we obtain $g=G(z(\lambda))\phi_{\Omega}\vpi$, where $G\in H^2_{\Omega}$. Respectively,
	\begin{equation}\label{eq233}
		k_{Sz}(\l,\l_0)=K_{\Omega}(z(\l),z(\l_0))\phi_{\Omega}(z(\l))\vpi(\l)
		\overline{\phi_{\Omega}(z(\l_0))\vpi(\l_0)}.
	\end{equation}
	By \eqref{217} and \eqref{eq232} we have
	$$
	\vpi^2\phi_{\Omega}^2=-i\frac 1{\cW}\prod_{j=1}^g\frac{z-c_j}{\sqrt{(z-a_j)(z-b_j)}}\frac{1}{2\sqrt{-z}}
	\frac{dz}{d\lambda}\times 2{\cW}\prod_{j=1}^g\frac{z-a_j}{z-c_j}=-i\Omega(z)\frac{dz}{d\lambda}.
	$$
	Thus \eqref{repom} and \eqref{eq233} imply \eqref{eq229}.
	
	It remains to explain the extremal property of the fixed half period. Consider the following family of functions
	\begin{equation*}
		\Omega_{\e_1,...,\e_g}(z)=\Omega(z)
		\prod_{k=1}^g\left(\frac{z-b_k}{z-a_k}\right)^{\frac{1-\e_k}{2}},\quad 
		\e_k=\pm 1.
	\end{equation*}
	Similarly to Lemma \ref{l212}, we have $2^g$ Hilbert spaces $H^2_{\Omega_{\e_1,...,\e_g}}$ with the reproducing kernels
	$$
	K_{\Omega_{\e_1,...,\e_g}}(z,z_0)=\frac{-\Omega_{\e_1,...,\e_g}^{-1}(z)+
		\Omega_{\e_1,...,\e_g}^{-1}(\bar z_0)}{2(z-\bar z_0)}.
	$$
	We define collection of outer functions
	$$
	\psi_{\e_1,\dots,\e_g}^2=\prod_{k=1}^g\left(\frac{z-b_k}{z-a_k}\right)^{\frac{1-\e_k}{2}}
	$$
	whose characters (for the already fixed $\Gamma$) form all possible  $2^g$ half periods on this group.
	In this case
	$$
	A_1^2(\G,\fj)=\{g=G(z(\l))\vpi\phi_{\Omega}\psi_{\e_1,\dots,\e_g}:\ G\in H^2_{\Omega_{\e_1,\dots,\e_g}}\}.
	$$
	Thus, we have to compare the value
	$$
	k_{Sz}(\l_0,\l_0;\G, \fj)=\frac{\Im \Omega_{\e_1,\dots,\e_g}(z_0)}{2\Im z_0|\Omega_{\e_1,\dots,\e_g}(z_0)|}
	\left|\frac{dz}{d\lambda}(\lambda_0)\right|
	$$
	for all half periods $\fj$ to choose the minimal one. We use the exponential representation
	$$
	\Omega_{\e_1,\dots,\e_g}(z)=Ce^{\int_\bbR\frac{1+xz}{x-z}\chi_{\e_1,\dots,\e_g}(x) 
		\frac{dx}{1+x^2}}, \ C>0,
	\ \chi_{\e_1,\dots,\e_g}(x):=\frac  1\pi {\arg\Omega_{\e_1,\dots,\e_g}(x)},
	$$
	due to which
	$$
	\frac{\Im \Omega_{\e_1,\dots,\e_g}(z)}{|\Omega_{\e_1,\dots,\e_g}(z)|}
	=\sin (I_{\e_1,\dots,\e_g}(z)), \quad I_{\e_1,\dots,\e_g}(z)=\int_\bbR\frac{\Im z}{|x-z|^2}\chi_{\e_1,\dots,\e_g}(x) dx.
	$$
	The minimal value in the last expression corresponds to the minimum between two extreme values
	$$
	\min_{\e_k=\pm 1}\sin (I_{\e_1,\dots,\e_g}(z))=\min\{\sin (I_{+}(z)), \sin (I_{-}(z))\}, \ I_{\pm}(z):=I_{\e_k=\pm1,\forall k}(z).
	$$
	For these two we have
	$$
	{\sin (I_{-}(z))-\sin (I_{+}(z))}=2\sin\left\{\frac 1 2\int_{\cup_{k=1}^g(a_k,b_k)}\frac{\Im z dx}{|x-z|^2} \right\}
	\cos\left\{\frac 1 2 \int_{\bbR_+}\frac{\Im z dx}{|x-z|^2}\right\}.
	$$
	Due to
	$
	\int_{\bbR}\frac{\Im z dx}{|x-z|^2}= \pi,
	$
	we get $\sin (I_{+}(z))<\sin (I_{-}(z))$. That is,
	the configuration $\e_k=1, \forall k,$ corresponds to the global minimum.
	
	Note that this extremal choice  of $\e_k$ corresponds to the exceptional case when
	$$
	\Re \Omega_{\e_1,\dots,\e_g}(z)\ge 0
	$$
	in the upper half-plane and therefore for all $z\in \cD$. Therefore, $\Omega(z)-\Omega(\bar z_0)$ has $g+1$,  that is, the \textit{maximal possible} number of zeros,  $\{\bar z_k\}_{k=0}^g$, in $\cD$. Respectively,
	$$
	\frac{\Omega(z(\l))-\Omega(\bar z_0)}{ \Omega(z(\l))+\Omega(\bar z_0)}=\prod_{k=0}^g b_{\bar z_k}(\l)
	$$
	and we obtain \eqref{eq230}.
\end{proof}

To describe reproducing kernels in the general case we introduce the following notations and definitions, see \cite[Chap. IIIa]{MTata2}, see also \cite{SY97}.

\begin{definition}\label{defdiv}
	Let $\cR$ denote the hyperelliptic Riemann surface
	$$
	\cR=\left\{P=(z,\Omega):\ \Omega^2=-\frac{1}{z}\prod_{k=1}^g\frac{z-a_j}{z-b_j}\right\}
	$$
	and $\bar \cR$ be its compactification. The upper sheet means the collection of points $\{P=(z,\Omega):\ \Re\Omega>0\}$ and we identify it with the domain $\cD=\bbC\setminus E_S$, where $\Omega=\Omega(z)$ is well defined. Thus, for a generic point on $\cR$ we can write $(z,1)$, $z\in\cD$, having in mind a point on the upper sheet and $(z,-1)$, $z\in\cD$, for a point on the lower sheet. 
	Note that $(a_j,\pm 1)$ (respectively $(b_j,\pm 1)$) denotes the same point on $\bar\cR$. The collection of points
	of the form
	$$
	D=\{(x_j,\e_j):\ x_j\in[a_j,b_j], \ \e_j=\pm 1\}_{j=1}^g
	$$
	we call a divisor. Topologically, they form a $g$-dimensional torus $D_{\cR}$. To the given $D$ we associate a multivalued function in $\cD$ (a generalization of \eqref{eq232})
	\begin{equation}\label{eq235}
		\phi_D(z(\l))=\sqrt{\frac 1 2 \prod_{j=1}^g \frac{(z(\l)-x_j)b_{c_j}(\l)}{(z(\l)-c_j)b_{x_j}(\l)}}\prod_{j=1}^g b_{x_j}^{\frac{1+\e_j}2}(\l),
	\end{equation}
	which can be extended on $\cR$. In this case $D$ corresponds to zeros of $\phi_D$. Note that poles  correspond to the divisor $\{(c_j,-1)\}_{j=1}^g$. The character generated by this function is denoted by $\alpha_D$.
	
	Let $T(z)=z\prod_{j=1}^g(z-a_j)(z-b_j)$ and $U_D(z)=\prod_{j=1}^g(z-x_j)$ so that
	$$
	\frac{U_D(z)}{\sqrt{-T(z)}}=\Omega(z)\prod_{j=1}^g\frac{z-x_j}{z-a_j}.
	$$
	has positive imaginary part in the upper half-plane of the upper sheet.  Let
	\begin{equation}\label{eq236}
		m_{\pm}^{D}(z)=m_{\pm}(z):=\frac{-\sqrt{-T(z)}\pm V_D(z)}{U_{D}(z)},
	\end{equation}
	where the polynomial $V_D(z)$, $V_D(0)=0 $, of degree $g$ is uniquely defined by the condition that on $\bar\cR$ the function $m_+(z)$
	has poles exactly at points forming the divisor $D$ (and, by construction, at infinity). Let us point out that both functions have positive imaginary values in the upper half-plane, and $\overline{m_+(x)}=-m_-(x)$, $x\in E_S$. Respectively,
	$D_*:=\{(x_j,-\e_j)\}_{j=1}^g$ is the divisor, which generates $m_-(z)$, i.e., $m_-^D(z)=m_+^{D_*}(z)$. 
\end{definition}

\begin{theorem}\label{th24n} For an arbitrary character $\alpha\in \G^*$ there exists a unique divisor $D$ such that
	$\alpha=\alpha_D$ with the character generated by $\phi_D$ \eqref{eq235}.
	Let $H^2_{m_+^D}$ be the space of meromorphic  functions $F(z)$ in $\cD$ such that $F(z(\l))\phi_D(z(\l))$ is of Smirnov class equipped with the scalar product
	$$
	\|F\|^2_{m_+^D}=\frac{1}{\pi}\int_E\frac{|F(x+ i0)|^2+|F(x-i0)|^2}{2}\Im\frac{U_D(x)}{\sqrt{-T(x)}} dx.
	$$
	Then the reproducing kernel of this space is of the form
	\begin{equation}\label{repmp}
		K_{m_+^D}(z,z_0)=\frac{m^D_+(z)-m_+^D(\bar z_0)}{{2}(z-\bar z_0)}, \quad z_0\not=x_k.
	\end{equation}
	Consequently, the reproducing kernel of the space $A_{1}^2(\G,\alpha\nu)$ is of the form
	\begin{equation*}
		k_{Sz}(\l,\l_0;\G,\alpha\nu)=K_{m_+^D}(z(\l),z(\l_0))\phi_D(z(\l))\vpi(\l)\overline{\phi_D(z(\l_0))\vpi(\l_0)}.
	\end{equation*}
\end{theorem}

\begin{proof}
	[of Theorem \ref{th24n} and Theorem \ref{th28}]
	Existence and uniqueness of $D$ for the given $\alpha$ follows from Abel-Jacobi inversion theorem.
	Since poles of $m_{\pm}^D$ complement each other in $\cD$, see \eqref{eq236}, we get by the Cauchy theorem
	$$
	\langle F,K_{m_+^D}(\cdot,z_0) \rangle=\frac 1{2\pi i}
	\oint_E  \frac{-m_-^D(x)-m_+^D(z_0)}{2(x-z_0)}F(x)\frac{U_{D}(x)}{\sqrt{-T(x)}}dx= F(z_0).
	$$
	
	\begin{figure}[htbp] 
		\begin{center}
			\includegraphics[scale=0.35
			]
			{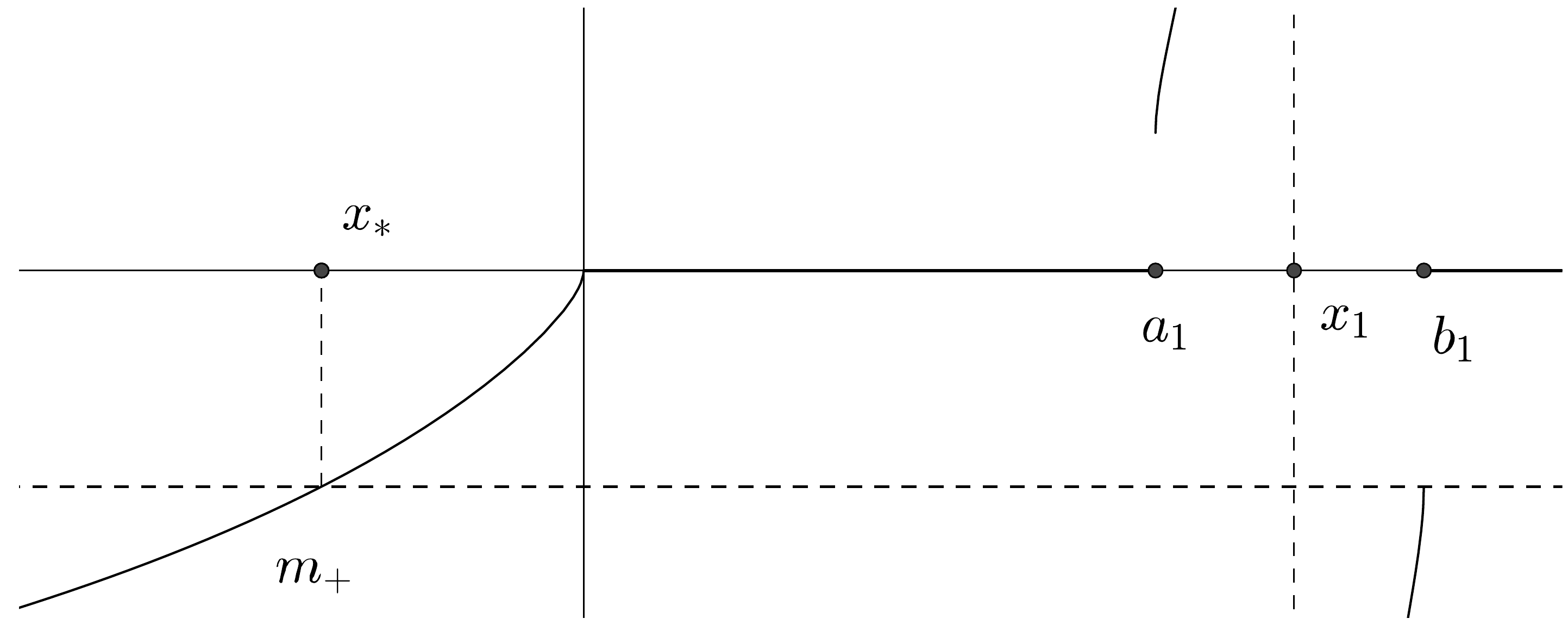}
			\caption{Change by a half period is required if $z_0\in(x_*,0)$}
			\label{mfunct}
		\end{center}
	\end{figure}
	It  remains to prove the last statement of Theorem \ref{th28}. Consider an elliptic case. 
	By \eqref{1eq227} we are interested in the reproducing kernel $k^\alpha_{Sz}(\lambda,\lambda_0)$ such that
	$\alpha^2=\beta$ and the corresponding $K_{m_+^D}(z,z_0)$ has zero in the gap $(a_1,b_1)$ (on the upper sheet).
	We use the representation \eqref{repmp}, assuming that $x_1\in(a_1,b_1)$ is a pole of $m_+^D$.
	Note that in this case $m_+^D(a_1)>0$ and $m_+^D(b_1)<0$.
	Therefore, see  a sketch of a graph of $m_+^D(z)$ in
	Fig. \ref{mfunct}, as soon as  $z_0<x_*$ the reproducing kernel $K_{m_+^D}(z,z_0)$ has indeed a zero in the gap $(a_1,b_1)$, but for $0>z_0>x_*$, $m_+^D(x)\not=m_+^D(z_0)$ for all $x\in[a_1,b_1]$ (the corresponding zero belongs to the lower sheet). Thus, in this range 
	a change of the half period is required (an extremal function corresponds to  another divisor $D$).
	
\end{proof}

\subsection{Ahlfors problem and Abel-Jacobi inversion}\label{ss23}  Recall that Abelian differentials on the Riemann surface $\cR$, see Definition \ref{defdiv}, form a $g$-dimensional linear space. We will fix a basis in this space in the form
\begin{equation}\label{l239}
	d\fw_k=\frac{Q_k(z)dz}{\sqrt{-T(z)}},\  \deg Q_k=g-1,\quad \int_{a_j}^{b_j}d\fw_k=\frac 1 2\delta_{k,j},\ k,j=1,\dots,g,
\end{equation}
where $\delta_{k,j}$ is the Kronecker symbol and integration is given along an interval on the upper sheet.
Note that $2\Re\int_0^zd\fw_k=\omega(z,E_k)$ in $\cD$, as before $\omega(z,F)$ denotes the harmonic measure in $\cD$, see \eqref{212}, and $E_k=E_S\setminus[0,a_k]$. Generally the Abel-Jacobi theorem sets a one-to-one correspondence between  the Jacobian variety Jac$(\cR)$ and the complex $g$-dimensional torus $\bbC^g/L_\cR$, where $L_{\cR}$ is the lattice generated by the matrix of periods of the Abelian integrals \eqref{l239}, see e.g. \cite[Chap. II, \S 2]{MTata1}.

Now, we are interested just in the \textit{real} part of the Abel-Jacobi inversion problem. Note that it plays a crucial role in the theory of finite gap self-adjoint/unitary operators, see e.g. \cite{AkhCongr, Apt84, Si11}. 
\begin{theorem}\label{lth25}
	Let us fix a base point $D_0=\{a_j\}_{j=1}^g$ in the collection of divisors $D_{\cR}$ on $\cR$ and define the map
	from $D_{\cR}$ to  the standard real torus $\bbR^g/\bbZ^g$ by
	\begin{equation}\label{laj}
		\alpha_k(D)=\sum_{j=1}^g\int_{a_j}^{(x_j,\e_j)}d\fw_k,  \ D=\{(x_j,\e_j)\}_{j=1}^g\in D_{\cR},\  
		\alpha=\{\alpha_k\}_{1}^g\in \bbR^g/\bbZ^g,
	\end{equation}
	with the integration along the interval on the lower or upper sheet depending on $\e_j$.
	This map is one-to-one.
\end{theorem}
Evidently, \eqref{laj} can be rewritten as
\begin{equation}\label{2laj}
	\alpha_k=\sum_{j=1}^g \frac {\e_j} 2\omega(x_j,E_k)\  \mod 1.
\end{equation}

In the context of the Chebyshev problematics, we have a quite similar, but actually different inversion problem.
For a fixed $\beta$ and $x_0<0$ the extremal function from Proposition \ref{lprop22} has the form of the Blaschke product
\begin{equation}\label{mexist}
	w_{x_0,\beta}(\l)=\prod b_{x_j}(\lambda),
\end{equation}
and for its character we get
\begin{equation}\label{3laj}
	\beta_k=\sum_{j=1}^g  \omega(x_j,E_k) \mod 1.
\end{equation}
Thus, the relation between $\beta$ and the divisor $\{(x_j,1)\}_{j=1}^g$ can not be reduced exactly to the standard Abel-Jacobi inversion (Theorem \ref{lth25}). Indeed, directly from \eqref{3laj} we can only get
$$
\alpha_k=\sum_{j=1}^g \frac 1 2 \omega(x_j,E_k) \mod 1,
$$
where $2\alpha_k=\beta_k$ is still defined up to a half period. If one of the solutions $\{\oc\alpha_k\}$ is fixed, the whole collection is of the form $\{\oc\alpha_k+\frac{1-{\delta_k}}{4}\}_{k=1}^g$, $\delta_k=\pm 1$. On this way to find $x_j$'s, we have to solve \eqref{laj} for all $2^g$ possible collections $\{\delta_k\}_{k=1}^g$ and after that to choose among all $2^g$  solutions the divisor $D=\{(x_j,\e_j)\}_{j=1}^g$ with $\e_j=1$, see \eqref{2laj}.

Therefore, in our context Theorem \ref{lth25} has to be substituted by the following claim.
\begin{proposition}\label{propgj1}
	Let $I_j$ be the interval $[a_j,b_j]$ with the identification of the endpoints and equipped with the corresponding topology of the unit circle. Let $\cI=\times_{j=1}^g I_j$ be the topological $g$-dimensional torus.
	Then the map $\cI\to\bbR^g/\bbZ^g$ given by
	\begin{equation}\label{7laji}
		\beta_k(\fX)=\sum_{j=1}^g\omega(x_j,E_k),  \ \fX=\{x_j\}_{j=1}^g\in \cI,\  
		\beta=\{\beta_k\}_{k=1}^g\in \bbR^g/\bbZ^g,
	\end{equation}
	is a homeomorphism.
\end{proposition}

\begin{proof}
	We fix $x_0$ in the complementary gap and for the given $\beta$ get an extremal Blaschke product  $w_{x_0,\beta}$ in the form \eqref{mexist}. Due to \eqref{3laj} we have existence  in \eqref{7laji}. 
	
	To get uniqueness, we note that any Blaschke product $w_{\beta}(\l)$ with the character $\beta$, of the form  \eqref{mexist} represents \textit{an inner part} of a reproducing kernel $k_{Sz}^{\alpha_0\beta}(\l,\l_0)$, $x_0=z(\l_0),$ for a certain $\alpha_0\in\G^*$. (Note in brackets, that the position of \textit{zeros of 
		reproducing kernels} deals precisely with the Abel-Jacobi inversion \eqref{laj}, see Theorem \ref{th24n}.)
	Since an outer part of a reproducing kernel is also a reproducing kernel we get
	$$
	k_{Sz}^{\alpha_0\beta}(\l,\l_0)=w_{\beta}(\lambda)\overline{w_{\beta}(\lambda_0)}k_{Sz}^{\alpha_0}(\l,\l_0).
	$$ 
	Thus, see \eqref{mineq},
	$$
	|w_{\beta}(\lambda_0)|^2=\inf_{\alpha\in \G^*}\frac{k_{Sz}^{\alpha\beta}(\l_0,\l_0)}{k_{Sz}^{\alpha}(\l_0,\l_0)}=
	\frac{k_{Sz}^{\alpha_0\beta}(\l_0,\l_0)}{k_{Sz}^{\alpha_0}(\l_0,\l_0)}.
	$$
	Due to  the uniqueness of the extremal function (in the normalization $w_{\beta}(\lambda_0)>0$) we get uniqueness in \eqref{7laji} (up to an identification of the gap endpoints).
\end{proof}

\section{Asymptotics: real case}

\begin{definition}
	We say that the comb $\Pi_J$ ($\Pi_S$) is $n$-regular ($\ell$-regular)  if $\frac{n}{\pi}\omega_k\in \bbZ$ for all $k$ (respectively,
	$\frac{\ell}\pi \omega_k\in\bbZ$).
\end{definition}

The Chebyshev alternation theorem and more general Markov's correc\-tions method allow to reveal the structure of the extremal polynomial or entire function in the real case. After that, it can be represented in terms of a conformal mapping on a suitable regular comb domain, see survey \cite{SY92}. 

\begin{theorem}\label{th31}
	Let  $x\in\bbR\setminus E_J$ ($x\in\bbR\setminus E_S$). For a given $n$ ($\ell$) there exists an $n$-regular ($\ell$-regular)  comb $\tilde\Pi_{J,n,x}$ ($\tilde \Pi_{S,\ell,x}$)
	such that the extremal function of Problem \ref{pJ} (Problem \ref{pS}) is given in terms of the corresponding conformal mapping
	\begin{equation}\label{eqrf}
		P_{n,x}(z)=\cos n\tilde\tau_{J,n,x}(z)\quad (F_{\ell, x}(z)=\cos\ell\tilde \tau_{S,\ell,x}(z))
	\end{equation}
	and
	\begin{equation*}
		A_n(x,E_J)=|P'_{n,x}(x)| \quad (A_\ell(x,E_S)=|F'_{\ell,x}(x)|).
	\end{equation*}
	In this case, the set $\tilde E_{J,n,x}$  ($\tilde E_{S,\ell,x}$), which corresponds to the base of the regular comb, 
	contains the initial set $E_J$ ($E_S$) and, on the other hand,  the original set contains the preimage of at least one of possibly two different points of the form
	$\frac {\pi k}{n}\pm 0$ ($\frac {\pi k}{\ell}\pm 0$) on the base of the regular comb for all $k=0,\dots, n$ ($k\in\bbZ_+$).
\end{theorem}

We will prove a counterpart of this theorem related to Problem \ref{pT}. For the Pell equation approach see \cite{L04}.
A comb $\Pi_T$ is called $n$-regular if $\frac{n}{2\pi}\omega_j\in\bbZ$ for all $j$, see Fig. \ref{comb}. 
\begin{theorem}
	Let  $e^{ix}\in\bbT\setminus E_T$. For a given $n$ there exists an $n$-regular comb $\tilde\Pi_{T,n,e^{ix}}$ such that the extremal function of Problem \ref{pT} is given in terms of the corresponding conformal mapping by
	\begin{equation*}
		P_{n,e^{ix}}(e^{iz})=e^{iz\frac n 2}\cos \frac n 2 \tilde\tau_{T,n,e^{ix}}(z)\quad
		\text{and}\quad
		A_n(e^{ix},E_T)=|P'_{n,e^{ix}}(e^{ix})|.
	\end{equation*}
	In this case, the set $\tilde E_{T,n,e^{ix}}$, which corresponds to the base of the regular comb, 
	contains the initial set $E_T$ and, on the other hand,  at least  one of  possibly two different points
	$e^{i\tilde\tau_{T,n,e^{ix}}^{-1}(\frac{2\pi k}{n}\pm0)}$ belongs to $E_T$ for all $k=0,\dots, n-1$.
\end{theorem}

\begin{proof} WLOG we assume that $x=x_0\in (a_0,b_0)$. Let $P(\z)=P_{n,e^{ix_0}}(\z)$ be an extremal polynomial. We represent it in the form 
	\begin{equation}\label{eq32}
		P(e^{iz})=e^{iz\frac n 2}F(z),
	\end{equation}
	clearly $F(z)$ is a periodic entire function, $F(z+2\pi)=(-1)^{n}F(z)$, and $F(x_0)=0$. Since an extremal polynomial is given up to multiplication by a unimodular constant, we can assume that $F'(x_0)>0$. After that we can substitute $F(z)$ by its symmetric part $\frac 1 2(F(z)+\overline{F(\bar z)})$ and we still get an extremal polynomial $P(e^{iz})$. Thus, we can assume that $F(z)$ is real on the real axis in the representation \eqref{eq32}. 
	
	We claim that $F(z)$ has no complex zeros, or, equivalently $P(\z)$ does not have zeros in $\bbC\setminus \bbT$. Indeed, let $F(z_0)=0$, $\Im z_0>0$. Note that $F(\bar z_0)=0$. Consider
	$$
	Q(\z)=P(\z)\left(1-\delta\frac{(\z-e^{ix_0})(1-\z e^{-ix_0})}{(\z-e^{iz_0})(1-\z e^{-i\bar z_0})}\right),\ \delta >0.
	$$ 
	Note that $Q(\z)$ is a polynomial of degree $n$ such that 
	\begin{equation}\label{eq33}
		Q(e^{ix_0})=0\quad \text{and}\quad 
		Q'(e^{ix_0})=P'(e^{ix_0}).
	\end{equation}
	In the same time, for $\z\in\bbT$
	$$
	\frac{(\z-e^{ix_0})(1-\z e^{-ix_0})}{(\z-e^{iz_0})(1-\z e^{-i\bar z_0})}=
	\frac{|\z-e^{ix_0}|^2}{|\z-e^{iz_0}|^2}.
	$$
	That is, for a sufficiently small but positive $\delta$ we have 
	\begin{equation}\label{eq34}
		\max_{\z\in E_T}|P(\z)|>\max_{\z\in E_T}|Q(\z)|.
	\end{equation}
	Thus $P(\z)$ was not an extremal polynomial.
	
	In a similar way we prove that all zeros of $F(z)$ are simple. Now we prove that between two (necessary real) consequent zeros of this function, say $z_1,z_2$, there is  a point $y\in (z_1,z_2)$ such that $|F(y)|=1$ and $e^{iy}\in E_T$. Assuming that on the contrary 
	$\{e^{iy};y\in (z_1,z_2)\}\cap E_T=\emptyset$, or that
	$\max_{y\in(z_1,z_2)\cap E_T}|F(y)|<1$,  we define the polynomial
	$$
	Q(\z)=e^{i\frac n 2 z}G(z), \quad G(z)=F(z)\left(1-\delta \frac{\sin^2\frac{z-x_0}{2}}{\sin\frac{z-z_1}{2}\sin\frac{z-z_2}{2}}\right).
	$$
	On the period we have to consider three regions: $I_1$ is a union of  small vicinities of points $z_1$ and $z_2$; the interval  $I_2=(z_1+\e,z_2-\e)$; and the remaining set $I_3$. On $I_1$, $|G(z)|$ is strictly less than one if $\delta$ is small. In $I_2$ the factor in brackets is greater than one, but there is no restriction on $|G(y)|$ if no one of the points $e^{iy}$, $y\in I_2$, belongs to $E_T$. In the second case, 
	$\max_{z\in I_2} |F(z)|$ is a fixed value, which is  less than one. So, a small $\delta>0$ can be chosen such that the product $|G(z)|$ is still less than one. On the remaining part $I_3$, 
	$\max |G(z)|<1$ due to the chosen correction factor. Since  \eqref{eq33} and \eqref{eq34} hold, we get a contradiction. 
	Note that the case $z_1=x_0$ ($z_2=x_0$) requires special, but basically the same consideration.  
	
	We can refer to general theorems \cite{MarOst75, SY92}, or, having in mind periodicity of $F(z)$, just to count the number of $\pm 1$ points (including multiplicity) on a period to conclude that all such points are real. Thus 
	$$
	\tilde \tau_T(z)=\frac 2{n}\arccos F(z), \quad \tilde\tau_T(x_0)=\frac\pi n,
	$$
	is well defined in the upper half-plane. Making inspection of the boundary behaviour we conclude that this is a conformal mapping on a suitable comb $\tilde \Pi_T$. Moreover, this comb is $n$-regular, according to our definition. 
	
	Since $|F(z)|\le 1$ for $e^{iz}\in E_T$ and generally $|F(z)|\le 1$ if and only if $\tilde\tau_T(z)\in\bbR$, we obtain that
	$$
	E_T\subset\tilde E_T=\{e^{iz}: \tilde\tau_T(z)\in\bbR\}.
	$$
	Zeros of $F(z)$ correspond to $z_k: \tilde \tau_T(z_k)=\frac {\pi+2\pi k} n$, $k\in\bbZ$.
	Therefore,  between each consequent pair $(z_{k-1}, z_{k})$ there is  a point $e^{iy}\in E_T$ such that $F(y)=\pm 1$. If the boundary of the domain $\tilde \Pi_T$ contains  a slit  with the base at  $\tilde\omega_k=\frac{2\pi k}{n}$  then $y$ corresponds either to the left or right limit point. Otherwise, this $y$ corresponds to a single point  $\tilde \omega_k$ on the boundary of $\tilde \Pi_{T}$.  
	
	Conversely, if we have an $n$-regular comb $\tilde \Pi_T$ and $\tilde \tau_T(z)$ is the comb function with the normalization $\tilde\tau_T(x_0)=\frac \pi n$,  then 
	$$
	F(z)=\cos \frac n 2\tilde \tau_T(z)
	$$
	is an analytic function in the upper half-plane, and real-valued on the real axis. Being extended by the symmetry principle in the lower half-plane, it represents an entire function of exponential type $\frac n 2$. Also $F(z+2\pi)=(-1)^nF(z)$. Thus
	$P(\zeta)$ of the form \eqref{eq32} is a polynomial of degree $n$. Every set, which contains  one of  possibly two different points
	$e^{i\tilde\tau^{-1}(\frac{2\pi k}{n}\pm0)}$  for all $k=0,\dots, n-1$, forms the so called maximal Chebyshev set of alternation, see e.g. \cite{SY92}. Due to the Chebyshev theorem $P(\z)$ is an extremal polynomial on an arbitrary $E_T$ containing the given set of alternation.
\end{proof}

\begin{remark}\label{rm31}
	The set $\tilde E_{T}=\tilde E_{T,n,e^{ix_0}}$ represents an extension of the set $E_T$,
	$$
	\tilde E_{T}=E_T\cup\{e^{iz};\ z\in \cup_{j=0}^n[u_j,v_j]\}, \quad [u_j,v_j]\subset[a_j,b_j].
	$$
	A simple analysis shows that there are the following three possibilities of a proper extension in the gap 
	$(a_j,b_j)$:  for a suitable $k_j\in\bbZ$
	\begin{itemize}
		\item[(a)] $[u_j,v_j]$ is an internal subinterval: there are two slits,  $\tilde h_{k_j}>0$, $\tilde h_{k_j+1}>0$,
		and
		$$
		\tilde \tau_T(a_j)=\frac{2\pi k_j}{n}-0,\quad
		\tilde \tau_T(b_j)=\frac{2\pi (k_j+1)}{n}+0;
		$$
		\item[(b)] a one-sided extension, say, $u_j=a_j$, $v_j<b_j$: there is a slit, $\tilde h_{k_j+1}>0$, and 
		$$
		\tilde \tau_T(b_j)=\frac{2\pi (k_j+1)}{n}+0\ 
		\left(\tilde \tau_T(v_j)=\frac{2\pi (k_j+1)}{n}-0\right) \ 
		\text{and}\  \tilde \tau_T(a_j)\ge\frac{2\pi k_j}{n}+0;
		$$
		\item[(c)] the gap is completely closed $u_j=a_j, v_j=b_j$: 
		$$
		\frac{2\pi k_j}{n}+0\le \tilde \tau_T(a_j)<\tilde \tau_T(b_j)\le\frac{2\pi (k_j+1)}{n}-0.
		$$
	\end{itemize}
	Due to \eqref{213} the harmonic measure in the origin of each additional arc $\{e^{iz}: z\in[u_j,v_j]\}$ in $\tilde \cD=\bar\bbC\setminus \tilde E_{T}$ is not more than $\frac 1 n$. That is, the length of each additional arc tends to zero as $n\to\infty$. Thus, ($E_T$ is fixed and $x_0\in (a_0,b_0)$) for sufficiently big $n$ the case (c) is not possible, and we always have case (a) for the chosen gap $(a_0,b_0)$, since $x_0\in[u_0,v_0]$.
\end{remark}

Now we can pass to the limit in $n$.

\begin{theorem}\label{thrmain}
	As soon as $z_0=z(\lambda_0)$ is real, solutions of Problems \ref{pJ}, \ref{pT}, \ref{pS} are given by
	\begin{align}
		\lim_{n\to\infty }\left\{e^{-nG(z_0,\infty)}\left|\frac{dz}{d\l}(\l_0)\right|A_n(z_0,E_J)-\frac 1 2\cA(\l_0,\alpha_{J,n})\right\}
		=0,\\
		\lim_{n\to\infty }\left\{e^{-nG(e^{iz_0},\infty)}\left|\frac{d z}{d\l}(\l_0)\right|A_n(e^{iz_0},E_T)-\frac 1 2\cA(\l_0,\alpha_{T,n})\right\}
		=0,\\
		\lim_{\ell\to\infty }\left\{e^{-\ell M(z_0)}\left|\frac{dz}{d\l}(\l_0)\right|A_\ell(z_0,E_S)-\frac 1 2\cA(\l_0,\alpha_{S,\ell})\right\}
		=0,\label{mainre3}
	\end{align}
	where $G(z,z_0)$, $G(\z,\z_0)$ and $M(z)$ are Green and Martin functions in the corresponding domains, and the characters are defined on the generators $\gamma_j\in\G$ by
	$$
	\alpha_{J,n}(\g_j)=e^{2i\omega_j n},\quad
	\alpha_{T,n}(\g_j)=e^{i\omega_j n},\quad
	\alpha_{S,\ell}(\g_j)=e^{2i\omega_j \ell},
	$$
	with $\omega_j$ corresponding to the combs $\Pi_J$, $\Pi_T$, $\Pi_S$ respectively, see Fig. \ref{comb}.
\end{theorem}

\begin{proof}
	We give a proof for entire functions, other cases are quite similar. Let $F_{\ell,x_0}(z)$, $\ell>0$, $x_0<0$, be the extremal function, see Theorem \ref{th31}. 
	Using compactness of $\Gamma^*$ we choose a  convergent sequence 
	\begin{equation}\label{eqchar3}
		\beta=\beta(\{\ell_k\})=\lim_{k\to\infty}\alpha_{S,\ell_k}.
	\end{equation}
	Using compactness of  the family 
	$$
	\{F_{\ell,x_0}(z(\l))e^{i\ell\cM(\l)}\}_{\ell>0},
	$$
	we choose an arbitrary subsequence of $\{\ell_k\}$ (but keeping the same notation) so that the following limit exist
	\begin{equation}\label{eq390}
		w(\l)=w(\l;\{\ell_k\})=\lim_{k\to\infty}F_{\ell_k,x_0}(z(\l))e^{i\ell_k\cM(\l)}.
	\end{equation}
	
	Now, consider the Martin function $\tilde M(z)=\tilde M_{\ell,x_0}(z)=\Im\tilde\tau_{S,\ell,x_0}$ of the domain $\bbC\setminus \tilde E_{S,\ell, x_0}$ in the original domain $\cD$. It is harmonic in the complement to the additional intervals 
	$\cup_{j=0}^g[u_j,v_j]=\tilde E_{S,\ell, x_0}\setminus E_S$. It is continuous in the whole $\cD$, but its normal derivative  has jumps on the union of these intervals. Recall that $\tilde\tau(x)=\tilde \tau_{S,\ell, x_0}(x)$ is real valued on $\tilde E_{S,\ell, x_0}$. By the 
	Cauchy-Riemann equations  we have
	\begin{equation*}
		\frac{\pd \tilde M}{\pd y}(x)=\frac{d\tilde\tau(x)}{dx}, \quad x\in \tilde E_{S,\ell, x_0}\setminus E_S.
	\end{equation*}
	Thus, in terms of the Green function $G(z,z_0)$ of $\cD$, we obtain
	\begin{equation}\label{eq38}
		\tilde M_{S,\ell, x_0}(z)=M(z)-\frac 1 \pi \int_{\tilde E_{S,\ell, x_0}\setminus E}G(z,x)d\tilde \tau_{S,\ell, x_0}(x).
	\end{equation}
	According to Remark \ref{rm31}, for a sufficiently large $\ell$, each additional interval is of the form (a) or (b) and we have
	\begin{equation}\label{eq39}
		\tilde \tau_{S,\ell, x_0}(u_j)-\tilde \tau_{S,\ell, x_0}(v_j)
		\begin{cases}
			= \frac{\pi}\ell,& \text{case (a)},\\
			\le \frac{\pi}\ell,& \text{case (b)}.
		\end{cases}
	\end{equation}
	Also recall that as $\ell\to\infty$ the $\ell$-depending system of intervals $[u^{(\ell)}_0,v^{(\ell)}_0]$ shrinks to the point $x_0$.
	We again choose a subsequence, keeping the same notations, so that
	$$
	\lim_{m\to\infty}u^{(\ell_{k})}_j=\lim_{m\to\infty}v^{(\ell_{k})}_j=x_j, \quad j=1,\dots,g,
	$$
	for some $x_j\in[a_j,b_j]$.
	Since $G(z,x)$ is continuous in $x$ and  $G(z,a_j)=G(z,b_j)=0$,
	by \eqref{eq38} and \eqref{eq39}, we obtain
	\begin{equation}\label{eq40}
		\lim_{m\to\infty}\ell_{k}(\tilde M_{S,\ell_{k}, x_0}(z)-M(z))=-\sum_{j=0}^g G(z,x_j).
	\end{equation}
	Now we go back to \eqref{eq390}. For $z=z(\l)$, by \eqref{eqrf}, we have
	$$
	|w(\l)|=\lim_{k\to\infty}e^{\ell_{k}(\tilde M_{\ell_{k}, x_0}(z)-M(z))}
	\frac{\left|1+e^{2i\ell_{k}\tilde\cM_{\ell_{k}, x_0}(\l)}\right|} 2
	=\frac 1 2 \prod_{j=0}^g|b_{x_j}(\l)|,
	$$
	and  $\beta=\prod_{j=0}^{g}\mu_{x_j}$. 
	Since $\{x_j\}$ with a suitable identifications, see Proposition \ref{propgj1}, corresponds to the extremal function of Problem \ref{pgA} for the given $\beta$, we conclude that
	$$
	\lim_{k\to\infty } e^{-\ell_k M(z_0)}\left|\frac{dz}{d\l}(\l_0)\right|A_{\ell_k}(z_0,E_S)=\frac 1 2\cA(\l_0,\beta)
	$$
	along the \textit{original} sequence $\{\ell_k\}$.
	Since $\beta$ is an arbitrary character of the form \eqref{eqchar3}, we get \eqref{mainre3}.
	
\end{proof}

\section{To make it complex}
To complement proofs given in the previous section, here we will discuss    extremal polynomials $P_n(z)=P_{n,z_0}(z)$  of Problem \ref{pJ}.
Obviously, $e^{ic}P_n(z)$, $c\in\bbR$, is also an extremal polynomial. So, a dual setting of the problem is the following: we fix an arbitrary non zero value of the derivative at $z_0$ and look for a polynomial $\tilde P_n(z)$ with the smallest maximum norm $\|\tilde P_n\|$ on $E_J$. In this case, $P_n(z)=\frac{\tilde P_n(z)}{\|\tilde P_n\|}$.
Thus, due to the Kolmogorov criterion, see e.g. \cite{Akh53}, $P_n(z)$ is extremal if and only if
\begin{equation}\label{al2-1}
	\inf_{x\in E_J:\ |P_n(x)|=1}\Re (x-\bar z_0)^2P_n(x)\overline{Q_{n-2}(x)}\le 0
\end{equation}
for an arbitrary polynomial $Q_{n-2}$ of degree $n-2$. 

Let
$$
\frac{z-\bar z_0}{z-z_0}P_n(z)=\Phi_n(z)+i\Psi_n(z), \quad Q_{n-2}(z)=\cX_{n-2}(z)+i\cY_{n-2}(z)
$$
be the decompositions of the corresponding  polynomials into the real and imaginary parts. Then \eqref{al2-1} is of the form
\begin{equation*}
	\inf_{x\in E_J:\ \Phi_n(x)^2+\Psi_n(x)^2=1}\{\Phi_n(x) \cX_{n-2}(x)+\Psi_n(x) \cY_{n-2}(x)\}\le 0.
\end{equation*}
Due to the symmetry properties it is enough to solve Problem \ref{pJ} for $z_0$ in the upper half-plane $\bbC_+$.
Evidently in this case $P_n(z)$ has all zeros, except for $z_0$, in the lower half-plane and
$$
\displaystyle
\left|\frac{i+\frac{\Psi_n(z)}{\Phi_n(z)}}{{i-\frac{\Psi_n(z)}{\Phi_n(z)}}}\right|\le 1, \quad z\in\bbC_+.
$$
In other words $-\frac{\Psi_n(z)}{\Phi_n(z)}$ has positive imaginary part in the upper half-plane. Due to the well known property zeros of these two polynomials interlace.

\begin{lemma}\label{4l22}
	Assume that  for two real polynomials with
	$-\Im \frac{\Psi_n(z)}{\Phi_n(z)}\ge 0$, $z\in\bbC_+$, the    set
	$X=\{x\in E_J:\ \Phi_n(x)^2+\Psi_n(x)^2=1\}$ coincides with the collection of all zeros of $\Phi_n(x)$. If
	$$
	\Phi_n(x)^2+\Psi_n(x)^2\le 1, \ x\in E_J,
	$$
	then 
	$$
	P_n(z)=\frac{z-z_0}{z-\bar z_0}(\Phi_n(z)+i\Psi_n(z))
	$$
	is an extremal polynomial of Problem \ref{pJ} for
	an arbitrary zero $\bar z_0$ of the complex polynomial $\Phi_n(z)+i\Psi_n(z)$.
\end{lemma}

\begin{proof}
	In this case, due to the Kolmogorov criterion we have to check that
	$$
	\inf_{x\in X}\Psi_n(x) \cY(x)\le 0
	$$
	for an arbitrary real polynomial $\cY(x)$, $\deg \cY=n-2$. Assume that there exists $\cY(x)$, which violates this property. Since zeros of $\Phi_n(z)$ and $\Psi_n(z)$ interlace, and $\Psi_n(x) \cY(x)>0$ for all $x\in X$, the polynomial $\cY(z)$ has $n-1$ zeros. That is, $\cY(z)$ is zero identically, and we get a contradiction.
\end{proof}

Let $\tilde E_{J,n}$ be an $n$-regular extension for the given set $E_J$ and $\tilde \tau_n$ be the corresponding comb-function. We define associated polynomials $\Phi_n$ and $\Psi_n$ by
\begin{equation}\label{aleq3}
	e^{-i n\tilde\tau_n(z)}=\cos n\tilde\tau_n(z)-
	i\sin n\tilde\tau_n(z)=\Psi_n(z)+
	\sqrt{
		\prod_{j=0}^g\frac{(z-u_j)(z-v_j)}{(z-a_j)(z-b_j)}
	}\Phi_n(z),
\end{equation}
and note that their zeros interlace.

\begin{theorem}\label{4th25}
	Assume that for the given extension $\tilde E_{J,n}$
	\begin{equation*}
		\tilde\rho_n^2:=-\sup_{x\in E_J}\prod_{j=0}^g\frac{(x-u_j)(x-v_j)}{(x-a_j)(x-b_j)}>0.
	\end{equation*}
	Let $\cZ_n(\rho)=\{z_j\}_{j=0}^g$ be the collection of points conjugated to the   zeros of
	\begin{equation}\label{al4}
		{\fp_n(z,\rho)}=\rho\Phi_n(z)+i\Psi_n(z),\quad \rho^2< \tilde\rho_n^2.
	\end{equation}
	Then 
	$$
	P_{n,z_j}(z)=\frac{z-z_j}{z-\bar z_j}\fp_n(z,\rho)
	$$
	is the Ahlfors polynomial with respect to $z_j\in\cZ_n(\rho)$ for the given set $E_J$.
\end{theorem}

\begin{remark}
	To keep all zeros of $\fp_n(z,\rho)$ defined in \eqref{al4} in the lower half-plane one has to choose $\rho<0$ if $u_0>a_0$ and $\rho>0$ if $v_0<b_0$ (note that the leading coefficient of $\Phi_n(z)$ is positive, due to the standard choice of the square root at infinity in \eqref{aleq3}).
\end{remark}

\begin{proof}[Proof of Theorem \ref{4th25}] We have
	$$
	\Psi_n(x)^2+\rho^2\Phi_n(x)^2\le\Psi_n(x)^2-\prod_{j=0}^g\frac{(x-u_j)(x-v_j)}{(x-a_j)(x-b_j)}\Phi_n(x)^2= 1
	$$
	for $x\in E_J$. Therefore we can use Lemma \ref{4l22}.

\end{proof}

\begin{lemma}\label{lemma27}
	Assume that along a subsequence
	\begin{equation}\label{assl27}
		\lim_{k\to\infty}{u_m^{(n_k)}}=\lim_{k\to\infty}{v_m^{(n_k)}}= x_m\in(a_m,b_m),\quad m=0,\dots,g.
	\end{equation}
	Let $\fX=\{x_m\}_{m=0}^g$ and
	\begin{equation}\label{rhor}
		\tilde \rho^2=\tilde \rho^2(\fX)=-\sup_{x\in E_J}\frac{U_{\fX}^2(x)}{T(x)}>0,
	\end{equation}
	where
	$$
	T(z)=\prod_{j=0}^g(z-a_j)(z-b_j),\quad U_{\fX}(z)=\prod_{j=0}^g(z-x_j).
	$$
	Then for $\rho^2< \tilde \rho^2$, $\rho(b_0-x_0)>0$, and $z=z(\l)$ we have
	\begin{equation}\label{cwidom}
		\lim_{k\to\infty}e^{in_k\tau_J(z)}\fp_{n_k}(z,\rho)=\frac 1 2\prod_{j=0}^g b_{x_j}(\l)\left(\rho
		\frac{\sqrt{T(z)}}{U_{\fX}(z)}+i\right).
	\end{equation}
\end{lemma}

\begin{proof}
	We use a counterpart of \eqref{eq40}, $z=z(\lambda)$,
	\begin{equation}\label{4eq4}
		\lim_{n\to\infty}e^{in_k(\tau_J(z)-\tilde\tau_{J,n_k}(z))}=\prod_{j=0}^g b_{x_j}(\l).
	\end{equation}
	Therefore, due to the definition \eqref{aleq3},
	\begin{eqnarray*}
		\lim_{k\to\infty}e^{in_k\tau(z)} \fp_{n_k}(z,\rho)=
		\lim_{k\to\infty}\left( i \frac{e^{in_k(\tau_J(z)-\tau_{J,n_k}(z))}+e^{in_k(\tau_J(z)+\tau_{J,n_k}(z))}} 2\right.
		\\
		+\left.\rho \sqrt{\prod_{j=0}^g\frac{(z-a_j)(z-b_j)}{(z-u^{(n_k)}_j)(z-v^{(n_k)}_j)}}
		\frac{e^{in_k(\tau_J(z)-\tau_{J,n_k}(z))}-e^{in_k(\tau_J(z)+\tau_{J,n_k}(z))}} 2
		\right).
	\end{eqnarray*}
	By \eqref{assl27} and \eqref{4eq4} we get \eqref{cwidom}.
\end{proof}

Recall that any character is uniquely defined by  its values on a system of free generators, $\beta_j=\beta(\gamma_j)$, $j=1,\dots,g$, see Fig. \ref{fund}. Thus, $\beta=\prod_{j=0}^g\mu_{x_j}$ is equivalent to
\begin{equation}\label{xs}
	\sum_{j=0}^g\omega(x_j,E_k)=\beta_k
\end{equation}
where
$E_k= E_J^k$, see \eqref{212}, and $\omega(z,E_k)$ is the harmonic measure of $E_k$ in $\cD$ w.r.t. $z\in\cD$, and according to Proposition \ref{propgj1} the set $\cI$ and $\Gamma^*\simeq\bbR^g/\bbZ^g$ are homeomorphic.

\begin{theorem}\label{cth42} Assume that $\alpha_J$ is in a generic position, that is, $\clos\{\alpha_J^n\}_{n\in\bbZ}=\Gamma^*$. Then there is an open set $\cV=\cV_1\times\cV_2\subset\Gamma^*\times \bbC_+$ such that
	\begin{equation}\label{limita}
		\lim_{k\to\infty} e^{-n_kG(z_0,\infty)}A_{J,n_k}(z_0)=Y(z_0,\beta):=\frac 1{2\Im z_0}e^{-\sum_{j=0}^gG(x_j,z_0)},
	\end{equation}
	where   $\{\beta, z_0\}\in\cV$, $\beta=\lim_{k\to\infty}\alpha_{J}^{n_k}$,
	is related to $\{x_0,\dots,x_g,\rho\}$ by  \eqref{xs} and
	\begin{equation}\label{zzero}
		-i\rho=\frac{U_{\fX}(z_0)}{\sqrt {T(z_0)}}.
	\end{equation}
\end{theorem}

\begin{proof}
	We follow to the line of the proof of Theorem \ref{thrmain}.
	According to Lemma \ref{lemma27}, we  fix a subsequence 
	$n_k$ such that $\alpha_J^{n_k}\to\beta\in\Gamma^*$.
	Comparing characters in  \eqref{cwidom}, we have \eqref{xs}. 
	We say that $x_0>a_0$ is regular for the given $\beta$ if for a solution of the system \eqref{xs}
	we have $x_j\in(a_j,b_j)$. If so, we can choose an interval $I$ around $x_0$ such that all points of this interval are regular ($x_j\in(a_j,b_j)$ depends continuously on $x_0$). Moreover
	$$
	\inf_{x_0\in I}\tilde\rho^2(\{x_j\}_{j=0}^g)>0.
	$$
	Going back to \eqref{cwidom}, for a sufficiently small $\rho^2:\tilde\rho_*<\rho<0$, we have a unique solution $z_0=z_0(x_0,\rho)$ of the equation \eqref{zzero},
	$$
	z_0\simeq x_0-i\rho\frac{\sqrt{T(x_0)}}{U_{\fX}'(x_0)}.
	$$
	
	To summarize: for an open set of characters $\beta\in\cV_1\subset\Gamma^*$  equations \eqref{xs} and \eqref{zzero} set  a one-to-one correspondence $z_0=z_0(x_0,\rho)$ on
	an open set  $\tilde \cV_2=I\times(\tilde\rho_*,0)$. $\cV_2\subset \bbC_+$ is defined as the image of $\tilde \cV_2$.
	
	Finally, we note that for an expression of the form
	$$
	g(z)=\frac{z-z_0}{z-\bar z_0} f(z), \quad f(\bar z_0)=0,
	$$
	we have
	$$
	|g'(z_0)|=\frac{|f(z_0)|}{2\Im z_0}.
	$$
	Thus, to get \eqref{limita}, we use \eqref{cwidom} and  
	$\overline{\frac{\sqrt{T(z_0)}}{U_{\fX}(z_0)}}=\frac{\sqrt{T(\bar z_0)}}{U_{\fX}(\bar z_0)}=-\frac i \rho$. 
\end{proof}

Let us represent \eqref{zzero} in terms of potential theory.

\begin{lemma}
	Let $\omega_{\bbC_+}(z, F)$ be the harmonic measure of $F\subset \bbR$ at $z\in\bbC_+$ in the upper half-plane. Then  \eqref{zzero} implies 
	\begin{equation}\label{hxs}
		\omega_{\bbC_+}(z_0,\cup_{j=0}^g[a_j,x_j])=\omega_{\bbC_+}(z_0,\cup_{j=0}^g[x_j,b_j]).
	\end{equation}
	Moreover, \eqref{hxs} means that $\frac{\sqrt{T(z_0)}}{U_{\fX}(z_0)}$ assumes a pure imaginary value.
\end{lemma}

\begin{proof}
	Due to the integral representation
	$$
	\frac{\sqrt{T(z)}}{U_\fX(z)}=i C e^{\frac 1 2\int_{\cup_{j=0}^g[a_j,x_j]}\frac {1+xz}{x-z}\frac{dx}{1+x^2}
		-\frac 1 2\int_{\cup_{j=0}^g[x_j,b_j]}\frac {1+xz}{x-z}\frac{dx}{1+x^2}}, \ C\in\bbR,
	$$
	the required condition $\arg \frac{\sqrt{T(z_0)}}{U_{\fX}(z_0)}=\pm \frac \pi 2$ (depending on sign of $C$) corresponds to
	$$
	\int_{\cup_{j=0}^g[a_j,x_j]}\frac {\Im z}{|x-z_0|^2}dx=\int_{\cup_{j=0}^g[x_j,b_j]}\frac {\Im z}{|x-z_0|^2}dx
	$$
	that is, to \eqref{hxs}.
\end{proof}

Theorem \ref{cth42} reduces   asymptotics in the complex  polynomial Ahlfors problem to the following generalized Abel-Jacobi inversion problem, compare subsection \ref{ss23}, particularly Proposition \ref{propgj1}.
\begin{problem}\label{gaji}
	For fixed $\beta\in\Gamma^*$ and $z\in\bbC_+$ solve the system
	\begin{eqnarray}
		\sum_{j=0}^g\omega(x_j,E_k)=\beta_k, &k=1,\dots,g,\label{cs1}\\
		\sum_{j=0}^g\left(\int_{a_j}^{x_j}+\int_{b_j}^{x_j}\right)\frac{dx}{|x-z|^2}=0,&\label{cs2}
	\end{eqnarray}
	where $x_j\in[a_j,b_j]$.
\end{problem}

\begin{remark}
	We substitute these values in \eqref{limita} to define, the generally speaking multivalued, function 
	$\Up(\l,\beta)=Y(z(\l),\beta)\left|\frac{dz}{d\l}\right|$
	responsible for the required asymptotics.
\end{remark}

\begin{proposition}
	Problem \ref{gaji} is locally solvable.
\end{proposition}

\begin{proof} We will check that the Jacobian of the system \eqref{cs1}-\eqref{cs2} does not vanish.
	To the polynomial $U_{\fX}(z)$ we associate the polynomial $V_{\fX}(z)$ of the form
	\begin{equation*}
		V_{\fX}(z)=\sum_{j=0}^g\frac{\sqrt{T(x_j)}}{U_{\fX}'(x_j)}\frac{U_{\fX}(z)}{z-x_j}.
	\end{equation*}
	By this definition 
	$$
	\fm_{\fX}(z):=\det \begin{bmatrix}
	1&\dots&1\\
	\vdots&\vdots&\vdots\\
	x_0^{g-1}&\dots&x_g^{g-1}\\
	\frac{\sqrt{T(x_0)}}{z-x_0}
	&\dots &
	\frac{\sqrt{T(x_g)}}{z-x_g}
	\end{bmatrix}=\prod_{0\le k<j\le g}(x_k-x_j)\frac{V_{\fX}(z)}{U_{\fX}(z)}.
	$$
	For this reason the Jacobian 
	\begin{equation*}
		\det\begin{bmatrix}
			1&\dots&1\\
			\vdots&\vdots&\vdots\\
			x_0^{g-1}&\dots&x_g^{g-1}\\
			\frac{\sqrt{T(x_0)}}{z-x_0}-\frac{\sqrt{T(x_0)}}{\bar z-x_0}
			&\dots &
			\frac{\sqrt{T(x_g)}}{z-x_g}-\frac{\sqrt{T(x_g)}}{\bar z-x_g}
		\end{bmatrix}=\fm_{\fX}(z)-\overline{\fm_{\fX}(z)}
	\end{equation*}
	does not vanish for all $z\in\bbC_+$ and $x_j\not=x_k$ for $j\not=k$.
\end{proof}

\begin{proposition}
	In the elliptic case, $g=1$,  Problem \ref{gaji} is globally solvable, but for some $\beta$ a solution is not unique.
\end{proposition}
\begin{proof}
	We map $\cD$ to the complement of the system of two arcs $E_T$ such that  $z_0\mapsto 0$. Let
	$$
	E_{T}=\{\z=e^{i x}:\ x\in[0,2\pi)\setminus (a_0,b_0,)\cup(a_1,b_1)\},\ \
	0=a_0<b_0<a_1<b_1<2\pi.
	$$
	Since the harmonic measure $\omega_{\bbC_+}(z_0,F)$ corresponds to the Lebesgue measure on $\bbT$, \eqref{hxs} corresponds to
	\begin{equation}\label{29e1}
		x_0+x_1=\frac{b_0+a_1+b_1}{2}:=c.
	\end{equation}
	WLOG we assume that $b_0>b_1-a_1$. As soon as $x_1\in (a_1,b_1)$, by \eqref{29e1} $x_0$ runs in 
	the interval $(\xi_-,\xi_+)$, 
	$\xi_+= c-a_1$, $\xi_-=c-b_1$, and we have
	$$
	\omega\left(e^{i\xi_+},E_2\right)\le\omega(e^{ix_0},E_2)+\omega(e^{ix_1},E_2)\le 1+
	\omega\left(e^{i\xi_-},E_2\right),\quad E_2=E_T\setminus E_T^1.
	$$
	Since $\omega\left(e^{i\xi_+}, E_2\right)<\omega\left(e^{i\xi_-}, E_2\right)$, \eqref{xs} is solvable for all
	$\beta\in \bbR/\bbZ$, but in the range $(\omega\left(e^{i\xi_+}, E_2\right),\omega\left(e^{i\xi_-}, E_2\right))$ a solution is not unique.
\end{proof}

\begin{figure}[htbp] 
	\begin{center}
		\includegraphics[scale=0.3]
		{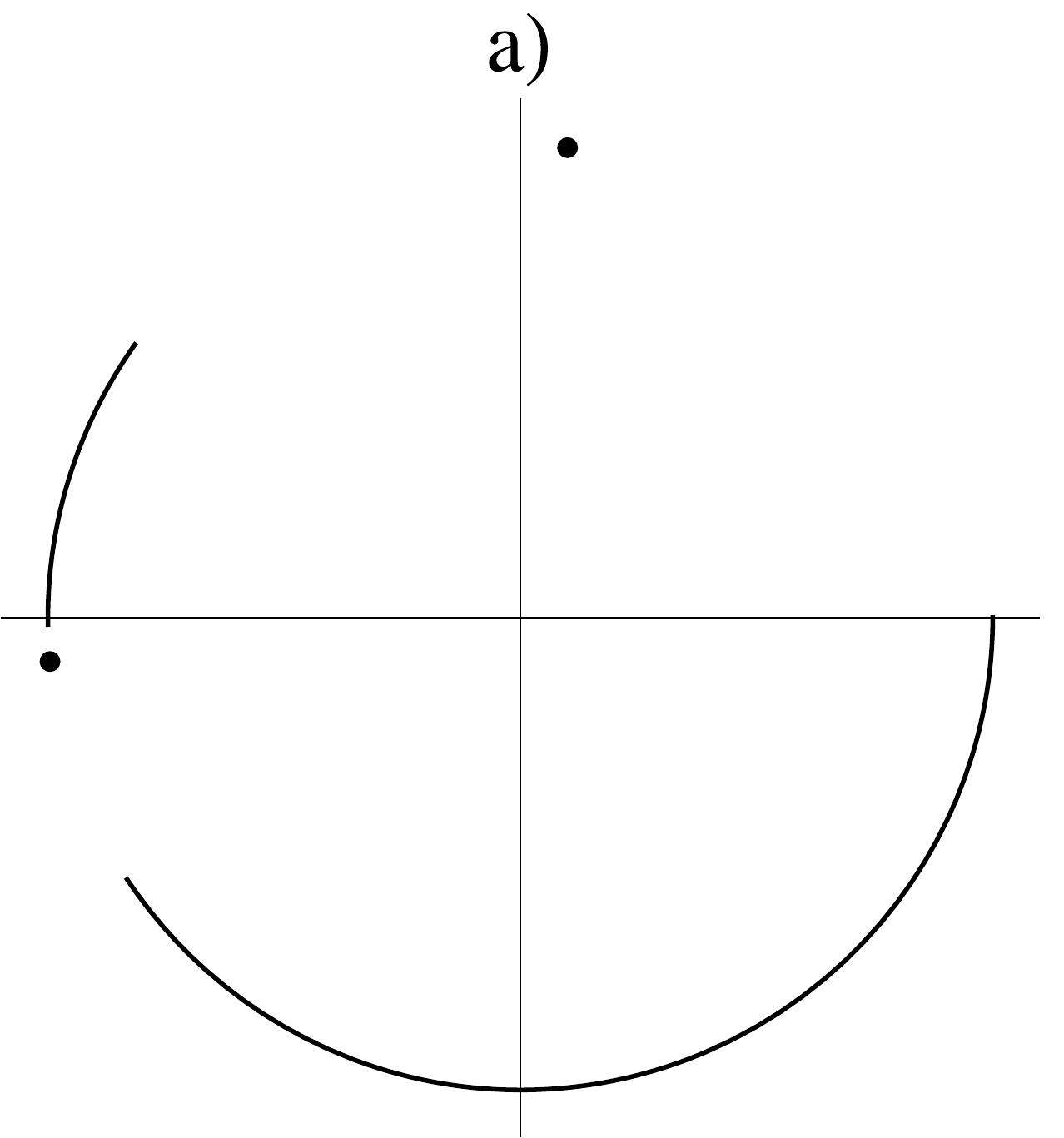}
		\includegraphics[scale=0.3]
		{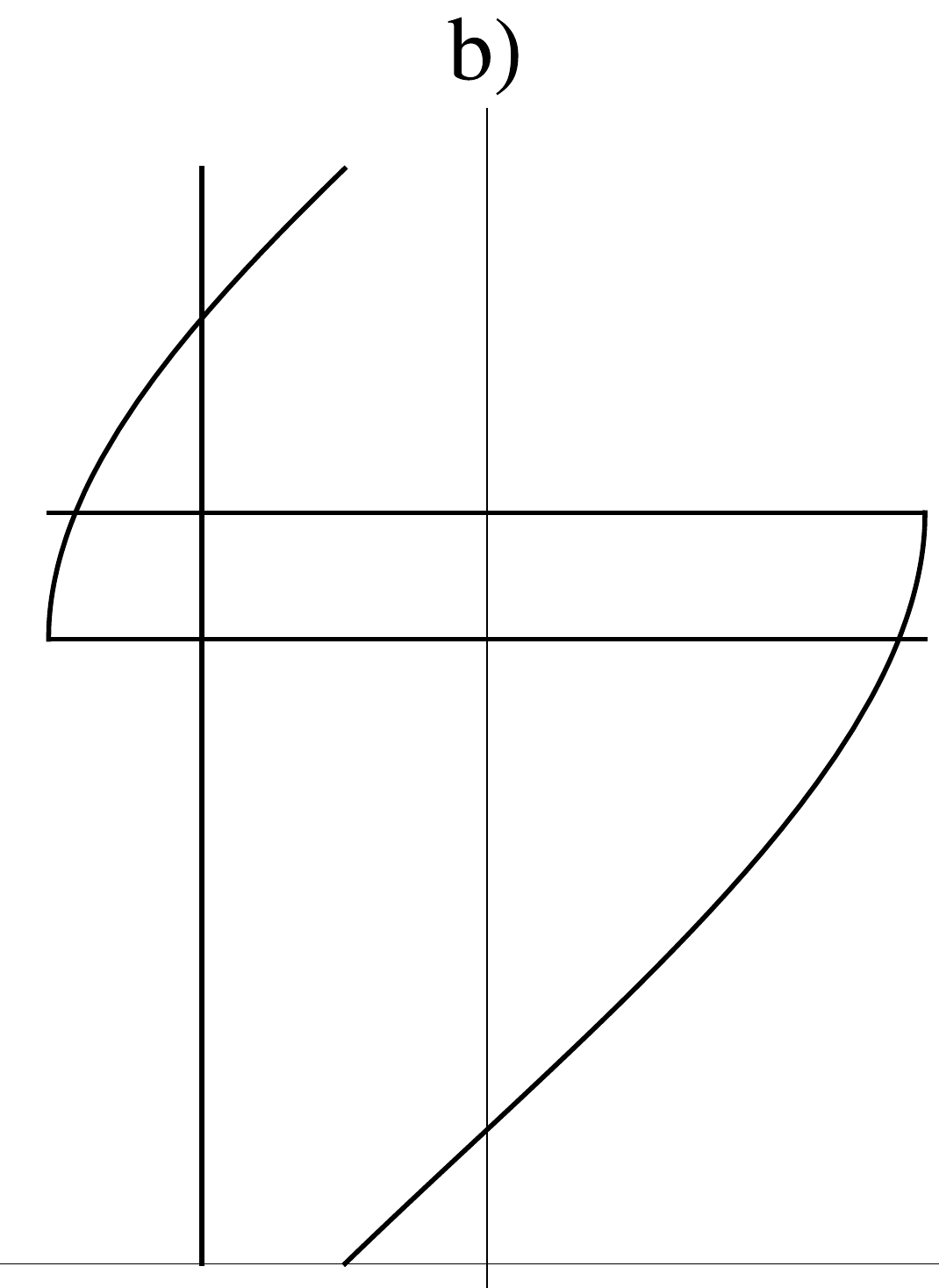}
		\includegraphics[scale=0.4]
		{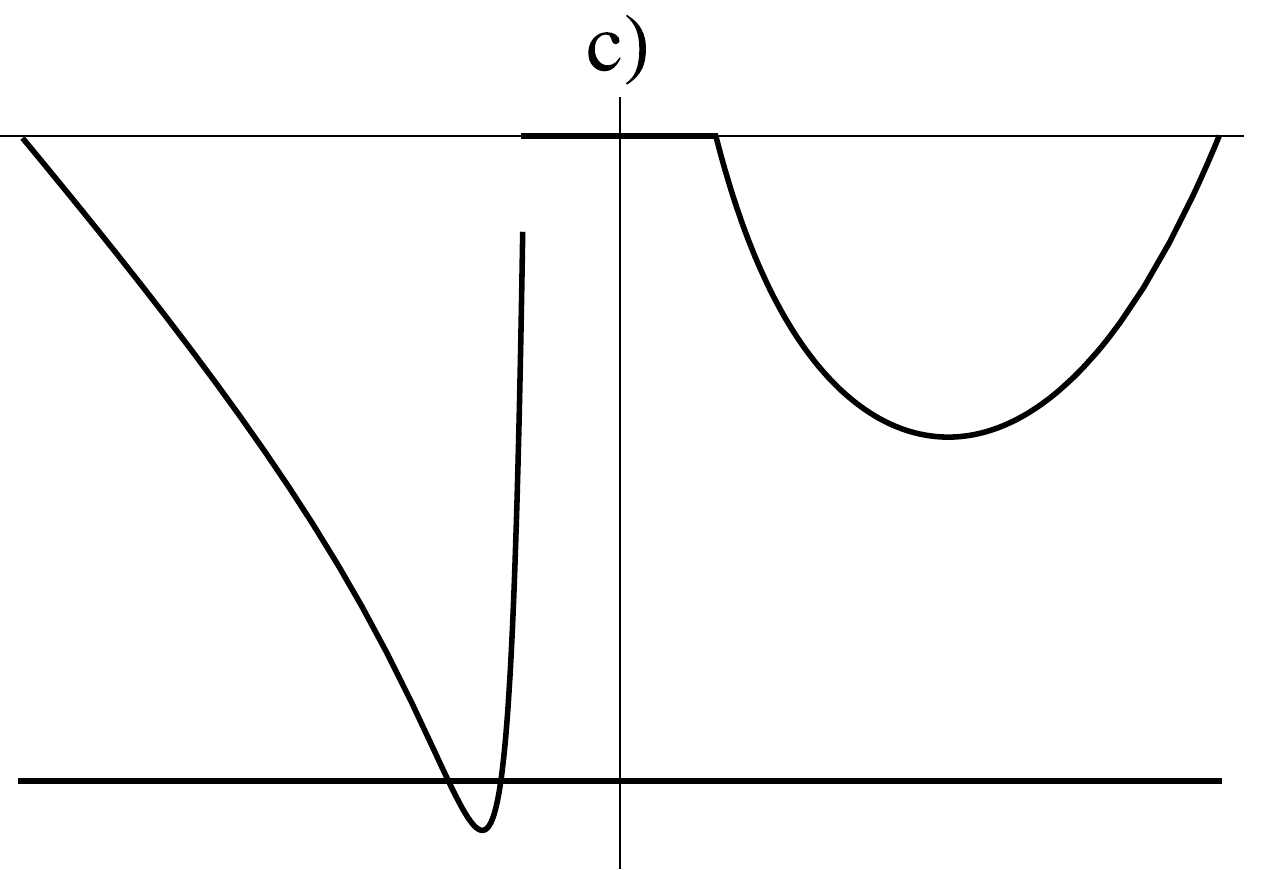}
		\caption{A modification of the shape of an extremal polynomial is required}
		\label{bif}
	\end{center}
\end{figure}

\begin{remark}
	Numerical experiments show that already in the elliptic case a certain bifurcation in the shape of an extremal polynomial may occur. In Fig. \ref{bif} we have the system of two arcs on $\bbT$,  see Plot (a).  Recall that $\zeta_0=0$ in this case and the values $e^{ix_0}$ and $e^{ix_1}$ are shown in the gaps. We vary  the parameter 
	$\beta$,
	$$
	\beta=\omega(e^{ i (c-x_1)},E_2)+\omega(e^{ix_1},E_2)\quad \mod 1,
	$$ 
	(the vertical axis in  Plot (b)) in the range, where a solution $x_1$ (on the horizontal axis) of Problem \ref{gaji} is unique. At a certain moment the value $\rho$ violates the condition $\rho^2\le\tilde\rho^2$, see \eqref{rhor} and Plot (c), where we graph the function $\frac{ {T(\z)}}{U^2_{\fX}(\z)}$, $\z\in E_T$ (the horizontal line corresponds to the level $-1/\rho^2$).
\end{remark}

\begin{remark}\label{30rem2}
	As it was demonstrated in the previous remark our function $Y(z,\beta)$  describes asymptotics \eqref{limita} for Problems \ref{pJ}, \ref{pT}, \ref{pS} only in a certain vicinity of the real axis.
	Note, however, that an analytic reproducing kernel is uniquely  defined by its diagonal
	\begin{equation}\label{30drk}
		k(\l_1,\l_2)=\sum\frac 1{n! m!}(\pd^n\bar\pd^m k)(\l_0,\l_0)(\l_1-\l_0)^n(\bar\l_2-\bar\l_0)^m.
	\end{equation}
	Therefore, in the case that  Conjecture \ref{mainconj} is correct, our Theorem \ref{cth42}, in fact, gives a complete information on the asymptotics due to the analytic continuation. Indeed, Theorem \ref{thrmain} provides asymptotics on the whole real axis in the $z$-plane, $\Up(\l,\beta)=\frac 1 2 \cA(\l,\beta)$, $z(\l)\in\bbR$.  To use \eqref{30drk} it is enough to have an extension of $\frac 1 2 \cA(\l,\beta)$ given by $\Up(\l,\beta)$ in an arbitrary small vicinity, $z(\l)\in\cV_2\subset \bbC_+$,
	\begin{equation*}\label{30drk2}
		k^{\alpha(\beta,\l_0)}(\l_1,\l_2)=\sum\frac 1{n! m!}(\pd^n\bar\pd^m \Up)(\l_0,\beta)(\l_1-\l_0)^n(\bar\l_2-\bar\l_0)^m,
		\quad z(\l_0)\in\bbR.
	\end{equation*}
	Let us recall, by the way, the condition \eqref{hstr} with respect to these partial derivatives.
\end{remark}

\begin{proposition} \label{pr43}
	For simply connected domains Problem \ref{gaji} is uniquely solvable. Moreover, 
	\begin{equation}\label{nnn}
		\Up(\l)=Y(z(\l))\left|\frac{dz}{d\lambda}\right|=\frac{2\sqrt{\l}\overline{\sqrt{\l}}}{(\lambda+\bar\lambda)(\sqrt{\lambda}+\overline{\sqrt{\lambda}})^2}.
	\end{equation}
\end{proposition}

\begin{proof}[Proof of Proposition \ref{pr43} and Theorem \ref{thintro}]
	Let 
	$\cD=\bbC\setminus \bbR_+$. In this case
	$$
	\frac{U_{x_0}(z_0)}{\sqrt{T(z_0)}}=\frac{-z_0+x_0}{\sqrt{-z_0}}=-i\rho\Rightarrow \Re{\sqrt{-z_0}}+x_0\frac{\Re\sqrt{-z_0}}{|z_0|}=0,
	$$
	thus we get a unique $x_0=-|z_0|$. By  \eqref{limita}, we have 
	$$
	Y(z_0)=\frac{1}{|z_0-\bar z_0|}\left|\frac{\sqrt{-z_0}-\sqrt{|z_0|}}{\sqrt{-z_0}+\sqrt{|z_0|}}\right|.
	$$
	We consider the right half-plane as the universal covering. Let $z_0=-\lambda^2$. Then 
	$$
	Y(z_0)=\left|\frac {1}{(\lambda-\bar\lambda)(\lambda+\bar\lambda)}\frac{\lambda-|\lambda|}{\lambda+|\lambda|}
	\right|=\frac{1}{(\lambda+\bar\lambda)(\sqrt{\lambda}+\overline{\sqrt{\lambda}})^2}.
	$$
	Since $|z'(\l)|=2|\l|$, we get \eqref{nnn}. 
	
	Since this solution is global the argument in the proof of Theorem \ref{cth42}
	are applicable to all $z_0\in\bbC_+$ and we have \eqref{eq5}.
	
\end{proof}

 \bibliographystyle{amsplain}
 \bibliography{lit}

\end{document}